\documentclass[a4paper]{amsart}

\usepackage[T1]{fontenc}
\usepackage{amssymb,mathrsfs,amsthm}
\usepackage{xy}
\usepackage{amsfonts}
\usepackage{amsmath}
\usepackage{fancyhdr}
\usepackage{dsfont}
\usepackage{caption}
\usepackage{graphicx}
\usepackage{hhline}
\usepackage{booktabs, cellspace}
\usepackage{longtable}
\usepackage{afterpage}
\usepackage{lscape}
\usepackage{tikz}
\usepackage{tikz-cd}
\usepackage{placeins}
\usepackage{hyperref}
\usepackage[
backend=biber,
style=alphabetic, maxbibnames=9
]{biblatex}

\newcommand{\bR}{{\mathbb{R}}}
\newcommand{\bC}{{\mathbb{C}}}

\newcommand{\bE}{{\mathbb{E}}}
\newcommand{\bF}{{\mathbb{F}}}
\newcommand{\bM}{{\mathbb{M}}}

\newcommand{\bS}{{\mathbb{S}}}
\newcommand{\bZ}{{\mathbb{Z}}}

\newcommand{\cA}{{\mathcal{A}}}
\newcommand{\Ext}{{\mathrm{Ext}}}
\newcommand{\F}{\mathrm{F}}
\newcommand{\E}{\mathrm{E}}

\newcommand{\kappabar}{\bar{\kappa}}
\newcommand{\ttwo}{\widetilde{2}}

\newcommand{\<}{\langle}
\newcommand{\h}{{\mathrm{H}}}
\newcommand{\BP}{\mathrm{BP}}
\newcommand{\MU}{\mathrm{MU}}
\newcommand{\Mod}{\mathrm{Mod}}
\newcommand{\Syn}{\mathrm{Syn}}
\newcommand{\syn}{\mathrm{syn}}
\newcommand{\cl}{\mathrm{cl}}
\newcommand{\Stable}{\mathrm{Stable}}
\newcommand{\Sp}{\mathrm{Sp}}

\makeatletter
\newcommand{\colim@}[2]{%
  \vtop{\m@th\ialign{##\cr
    \hfil$#1\operator@font colim$\hfil\cr
    \noalign{\nointerlineskip\kern1.5\ex@}#2\cr
    \noalign{\nointerlineskip\kern-\ex@}\cr}}%
}
\newcommand{\colim}{%
  \mathop{\mathpalette\colim@{\textstyle}}\nmlimits@
}
\makeatother

\newtheorem{theorem}{Theorem}[section]
\newtheorem*{theorem*}{Theorem}
\numberwithin{theorem}{section}

\newtheorem{proposition}[theorem]{Proposition}
\newtheorem{lemma}[theorem]{Lemma}
\newtheorem{corollary}[theorem]{Corollary}
\theoremstyle{definition}
\newtheorem{definition}[theorem]{Definition}
\newtheorem{remark}[theorem]{Remark}
\theoremstyle{definition}
\newtheorem{exmp}[theorem]{Example}

\tikzcdset{
  mysep/.style={row sep=#1, column sep=0.8em},
}
\renewcommand{\arraystretch}{1.5}
\title[$H\mathbb{F}_2$-synthetic homotopy groups of $tmf$]
{\textbf{$H\mathbb{F}_2$-synthetic homotopy groups of \\ topological modular forms}}
\author{Peter J. Marek}

\begin{document}

\begin{abstract}
To any Adams-type spectrum $E$, Pstr\k{a}gowski produced a symmetric monoidal stable $\infty$-category $\Syn_E$ whose objects are, in a sense, ``formal Adams spectral sequences''. $\Syn_E$ comes equipped with a lax symmetric monoidal functor $\nu_E:Sp\to \Syn_E$ from classical spectra, which embeds $\Sp$ fully and faithfully in $\Syn_E$, and is a category with a natural notion of bigraded homotopy groups. The bigraded homotopy groups $\pi_{*,*}\nu_EX$ systematically record information about the homotopy groups $\pi_*X$ and the $E$-Adams spectral sequence of $X$. In this paper, we compute the $\nu_{\h\bF_2}\h\bF_2$-Adams spectral sequence of $\nu_{\h\bF_2}tmf_2^{\wedge}$, synthetic versions of hidden $2$-, $\eta$-, $\nu$-, and $\kappabar$-extensions, and use this to deduce information about the homotopy ring structure of $\pi_{*,*}\nu_{\h\bF_2}tmf_2^{\wedge}$.
\end{abstract}

\maketitle

\section{Introduction}

The computation of the stable homotopy groups of spheres is a central and important problem in algebraic topology. In recent years, substantial progress has been made in understanding them at the prime $p=2$ \cite{IWX23} using the stable $\infty$-category of $2$-complete cellular motivic spectra $(\Sp_\bC)^{\wedge}_2$.

\bigskip

A primary tool used for this is the $C\tau$-philosophy in $(\Sp_\bC)^{\wedge}_2$. In short, the $C\tau$-philosophy entails using a particular $\bE_{\infty}$-ring spectrum $C\tau\in(\Sp_\bC)^{\wedge}_2$ and an isomorphism of spectral sequences \cite[Theorem 1.17]{GWX21}
\begin{equation*}
\begin{tikzcd}
\Ext_{\BP_*\BP/I}^{s,2w}(\BP_*/I, I^{a-s}/I^{a-s+1}) \ar[r,"\cong"] \ar[d, Rightarrow,"\textbf{algebraic NSS}"'] & \Ext_{\mathcal{A}^{mot}_{*,*}}^{a,2w-s+a,w}(\bM_2,H_{*,*}(C\tau)) \ar[d, Rightarrow,"\textbf{motivic Adams SS}"] \\
\Ext_{\BP_*\BP}^{s,2w}(\BP_*,\BP_*) \ar[r,"\cong"'] & \pi_{2w-s,w}C\tau
\end{tikzcd}    
\end{equation*}
to produce purely algebraic calculations of Adams differentials in the motivic Adams spectral sequence
\begin{equation*}
\begin{tikzcd}
 \Ext_{\mathcal{A}^{mot}_{*,*}}^{*,*,*}(\bM_2,\bM_2) \ar[r,Rightarrow] & \pi_{*,*}(\bS_{\bC})
\end{tikzcd}   
\end{equation*}
for the 2-complete motivic sphere spectrum. This then gives information about the classical 2-complete sphere spectrum via the Betti realization functor $\beta:\Sp_{\bC}\to\Sp$. (See \cite[Section 1.3]{GWX21} for more on the $C\tau$-philosophy).

\bigskip

For any Adams-type spectrum $E$, Pstr\k{a}gowski \cite{Pst23} produced a stable $\infty$-category $\Syn_E$, the category of $E$-\emph{synthetic spectra}, which generalizes the $C\tau$-behavior seen in $\Sp_{\bC}$ without the use of motivic methods. There are many reasons to study synthetic spectra:

\begin{itemize}
     \item One modern approach to understanding stable homotopy groups of spheres is to study the sphere spectrum (unit object) in a ``neighboring'' stable $\infty$-category and compare it to $\Sp$. Standard examples include $G$-equivariant spectra, $\bR$- and $\bC$-motivic spectra, etc. $\Syn_E$ offers another setting in which to do this.
    \item $\Syn_E$ comes equipped with an embedding $\nu_E: \Sp\to\Syn_E$ such that the $\nu_E E$-Adams spectral sequence for the synthetic spectrum $\nu_EX$ is equivalent to a synthetic $\tau$-Bockstein spectral sequence for $\nu_EX$. Bockstein spectral sequences are often easier to work with and prove things about in practice. In addition, this spectral sequence systematically keeps track of all the data of differentials and hidden extensions in the $E$-Adams spectral sequence of $X$. We discuss more details of this in Section~\ref{synspectraoverviewsectino}.
    \item Synthetic spectra can themselves be used as a tool to attack classical Adams spectral sequence calculations, particularly hidden extension calculations (see \cite{Bur20} and Remark~\ref{CtauLESremark} of this paper for examples).
\end{itemize}

In chromatic homotopy theory, the connective topological modular forms spectrum $tmf$ and its variants \cite{DFHH14}
are key spectra for understanding height 2 phenomena. In particular, $tmf$ is a good approximation to $\bS$ in the sense that $tmf$ is an $\bE_{\infty}$-ring spectrum and the ring map $\pi_*(\bS)\to\pi_*(tmf)$ induced by the unit $\bS\to tmf$ detects many elements of $\pi_*(\bS)$.

\bigskip

Because of the close relationship between $\Syn_E$ and the $E$-Adams spectral sequence, we hope to understand more about $\pi_*(\bS_2^{\wedge})$ by understanding the bigraded homotopy ring $\pi_{*,*}(\nu_{\h\bF_2}\bS_2^{\wedge})$ and the ring map $\pi_{*,*}(\nu_{\h\bF_2}\bS_2^{\wedge})\to \pi_{*,*}(\nu_{\h\bF_2}tmf_2^{\wedge})$ induced by the synthetic unit map $\nu_{\h\bF_2}\bS_2^{\wedge}\to\nu_{\h\bF_2}tmf_2^{\wedge}$. Forthcoming work of Burklund, Isaksen, and Xu aims to do part of this program by studying a synthetic Adams spectral sequence computing $\pi_{*,*}(\nu_{\h\bF_2}\bS_2^{\wedge})$ and using it to obtain specific information about $\pi_*(\bS_2^{\wedge})$ when $*\geq 70$.

\bigskip

Let $tmf:=tmf_2^{\wedge}$ and $\nu\h\bF_2:=\nu_{\h\bF_2}\h\bF_2$. In this paper, we study the bigraded synthetic homotopy ring of $\nu_{\h\bF_2}tmf$. Our main tool is the $\nu\h\bF_2$-Adams spectral sequence, which by \cite[Thm. 9.19]{BHS19} is a ``deformed'' version of the classical $\h\bF_2$-Adams spectral sequence. Most of the methods we use to compute this spectral sequence are essentially synthetic versions of the ones employed by \cite{BR21} in their computation of the $\h\bF_2$-Adams spectral sequence of $tmf$.

\bigskip

In Section~\ref{synspectraoverviewsectino}, we overview some basic properties of the category $\Syn_E$ and the $\nu\h\bF_2$-Adams spectral sequence which we will use in our calculation.

\bigskip

The $\h\bF_2$-Adams ${}_{\cl}\E_2^{*,*}$-page for $tmf$ is an $\bF_2$-algebra with 13 generators and 54 relations, originally described by \cite{SI67}. \cite{BR21} produce a description of the ${}_{\cl}\E_r^{*,*}$-pages instead as $R_i:=\bF_2[g,w_1,w_2^{2^i}]$-modules for some particular generators \linebreak $g,w_1,w_2$. In Section~\ref{AdamsSScalcdetailssection}, we review their calculation and calculate the ${}_{\syn}\E_r^{*,*,*}$-pages of the $\nu\h\bF_2$-Adams spectral sequence for $\nu_{\h\bF_2}tmf$. Let $R_i':=\bF_2[\tau,g,w_1,w_2^{2^i}]$ . Our main result in this section is the following:

\begin{theorem}(Section~\ref{AdamsSScalcdetailssection})
The ${}_{\syn}\E_r^{*,*,*}$-pages of the $\nu\h\bF_2$-Adams spectral sequence for $\nu_{\h\bF_2}tmf$ are completely computed. The $R_0'$-module structure of ${}_{\syn}\E_2^{*,*,*}$, the $R_1'$-module structure of ${}_{\syn}\E_3^{*,*,*}$, and the $R_2'$-module structures of ${}_{\syn}\E_4^{*,*,*}$ and ${}_{\syn}\E_5^{*,*,*}={}_{\syn}\E_{\infty}^{*,*,*}$ are recorded in Table~\ref{E2pagetable}, Table~\ref{E3pagetable}, Table~\ref{E4pagetable}, and Table~\ref{Einfpagetable} respectively.    
\end{theorem}

In Section~\ref{multstructsection}, we use the computation of ${}_{\syn}\E_{\infty}^{*,*,*}$ to study multiplicative structure of $\pi_{*,*}\nu_{\h\bF_2}tmf$. We first determine $\pi_{0,*}\nu_{\h\bF_2}tmf$ (Prop.\ref{pi0prop}). The homotopy groups $\pi_{*,*}\nu_{\h\bF_2}tmf$ are naturally a $\pi_{0,*}\nu_{\h\bF_2}tmf$-algebra and we determine the indecomposable algebra generators:

\begin{theorem}(\ref{synEinfalgelements})
As a $\bZ_2^{\wedge}[\tau,\ttwo]/(\tau\ttwo=2)$-algebra, $\pi_{*,*}\nu_{\h\bF_2}tmf$ is generated by 44 indecomposable elements. These elements (with the addition of $\ttwo$) are listed in Table~\ref{Einfalggentable} together with the elements of ${}_{\syn}\E_{\infty}^{*,*,*}$ which detect them. 
\end{theorem}

Let $\pi_{k,w}:=\pi_{k,w}\nu_{\h\bF_2}tmf$. We focus our attention on multiplicative structure involving the generators $\tau\in\pi_{0,-1}$, $\ttwo\in\pi_{0,1}$, $\eta\in\pi_{1,2}$, $\nu\in\pi_{3,4}$, $\kappabar\in\pi_{20,24}$, $B\in\pi_{8,12}$, and $M\in\pi_{192,224}$. We first determine hidden extensions by these elements:

\begin{theorem}(\ref{nohidextthm})
There are no hidden $\tau$-, $B$-, or $M$-extensions in $\pi_{*,*}(\nu_{\h\bF_2}tmf)$.    
\end{theorem}

\begin{theorem}(\ref{mainhidextthm})
All possible hidden $\ttwo$-, $\eta$-, $\nu$-, and $\kappabar$-extensions of $\pi_{*,*}(\nu_{\h\bF_2}tmf)$ are computed and recorded in Table~\ref{ttwoextensiontable}, Table~\ref{etaextensiontable}, Table~\ref{nuextensiontable}, and Table~\ref{kappabarextensiontable}, respectively.
\end{theorem}

With the exception of a few other relations, this turns out to be enough to describe the full multiplicative structure induced by these elements. Let $N_{[0,192)}\subset\pi_{*,*}\nu_{\h\bF_2}tmf$ denote the $\bZ_2^{\wedge}[\tau,\ttwo,\kappabar,B]/(\tau\ttwo=2)$-submodule generated by all classes in topological degrees $0\leq k <192$:

\begin{theorem}(\ref{Mfreethem})
As $\bZ_2^{\wedge}[\tau,\ttwo,\kappabar,B,M]/(\tau\ttwo=2)$-modules, there is an isomorphism 
\begin{equation*}
    N_{[0,192)}\otimes \bZ_2^{\wedge}[M]\xrightarrow{\cong}\pi_{*,*}\nu_{\h\bF_2}tmf.
\end{equation*}
\end{theorem}

\begin{theorem}(\ref{fullmultthm})
The bigraded homotopy groups $\pi_{*,*}(\nu_{\h\bF_2}tmf)$, as a module over $\bZ_2^{\wedge}[\tau,\ttwo,\eta,\nu,\kappabar,B,M]/(\tau\ttwo=2)$, are generated by 59 elements, recorded in Table~\ref{homotopymodgentable}, and subject to the relations described by Figures~\ref{chart1}-\ref{chart8}, Table~\ref{kappabarextensiontable}, Table~\ref{otherrelationstable}, Theorem~\ref{Btautorsionthm}, and Theorem~\ref{Mfreethem}.
\end{theorem}

In Appendix~\ref{appendix} we include Adams charts of ${}_{\syn}\E_{\infty}^{*,*,*}$ and of hidden extensions and relations in $\pi_{*,*}(\nu_{\h\bF_2} tmf)$, tables of the $R_i'$-generators of ${}_{\syn}\E_{r}^{*,*,*}$-pages, tables of the indecomposable algebra generators of ${}_{\syn}\E_{\infty}^{*,*,*}$ and $\pi_{*,*}(\nu_{\h\bF_2} tmf)$, a table of the $\bZ_2^{\wedge}[\tau,\ttwo,\eta,\nu,\kappabar,B,M]/(\tau\ttwo=2)$-module generators of $\pi_{*,*}(\nu_{\h\bF_2} tmf)$, tables of all $\ttwo$-, $\eta$, $\nu$-, and $\kappabar$-hidden extensions, and a table of miscellaneous relations.

\subsection{Acknowledgements} I thank Dan Isaksen for suggesting this calculation and for his comments on my Adams charts and a draft of this paper. Thanks to Dexter Chua, Piotr Pstr\k{a}gowski, and Robert Burklund for helpful conversations about synthetic spectra and to Dexter Chua and Hood Chatham for showing me how to use their sseq Ext resolver and GUI, which was indispensable for this calculation. Thanks to John Rognes and Bob Bruner for helpful conversations and for writing a monolith of a book on the classical Adams spectral sequence for $tmf$. Thanks to Noah Riggenbach for helpful comments on a draft of this paper. Thanks to an anonymous referee, who suggested several useful comments for improving this paper. Lastly, thanks to my advisor Michael Mandell for his support and patience answering any questions I had throughout this project.

\section{Overview of Synthetic Spectra}

\label{synspectraoverviewsectino}

In this section, we summarize known definitions and theorems about synthetic spectra. Most of the material in this section comes from \cite{Pst23} and from Section 9 and Appendix A of \cite{BHS19}.

\begin{theorem} \text{\cite{Pst23}}.
Let $E\in \Sp$ be an Adams-type spectrum (see \cite[Definition 3.13]{Pst23}); e.g. $E=\h\bF_p,\BP,\MU$, etc. There is a presentable symmetric monoidal stable $\infty$-category $\Syn_E$ and a commuting diagram of functors

\begin{equation*}
\begin{tikzcd}[column sep=small]
& \Syn_E \arrow[dr,"\tau^{-1}"] & \\
\Sp \arrow[rr,"id"'] \arrow[ur,"\nu_E"] & & \Sp
\end{tikzcd}  
\end{equation*}

 satisfying the following:
 
 \begin{enumerate}
     \item $\nu_E$ is lax symmetric monoidal and preserves filtered colimits.
     \item $\tau^{-1}$ is symmetric monoidal and preserves colimits. 
    
    
 \end{enumerate}
 
\end{theorem}

$\Syn_E$ is the category of $E$-\textit{synthetic spectra} (or just \textit{synthetic spectra} if $E$ is implicit). We will often write $\nu$ instead of $\nu_E$ if $E$ is implicit. The functor $\tau^{-1}$ we refer to as $\tau$-\textit{localization}.

\begin{corollary}
If $X$ is an $\bE_{\infty}$-ring spectrum, then $\nu X$ is a synthetic $\bE_{\infty}$-ring spectrum.
\end{corollary}

\begin{proof}
This follows directly from the fact that $\nu$ is lax symmetric monoidal.
\end{proof}

In $\Syn_E$, there are bigraded spheres which give rise to bigraded suspensions and bigraded homotopy groups:

\begin{definition} \textit{\cite[Section 4.1]{Pst23}}.
The bigraded sphere spectrum $\bS^{t,w}$ is defined to be $\Sigma^{t-w}\nu\bS^w$ and the monoidal unit is $\bS^{0,0}$. For synthetic spectra $X,Y$, bigraded suspension $\Sigma^{t,w}X$ is defined to be $\bS^{t,w}\otimes X$ and the bigraded homotopy classes of maps $[Y,X]_{t,w}$ is defined to be $\pi_0\mathrm{map}(\Sigma^{t,w}Y,X)$. In particular, we write $\pi_{t,w}(X)$ for $[\bS^{0,0},X]_{t,w}=[\bS^{t,w},X]$.
\end{definition}

\label{Def.2.4}

A key feature of synthetic spectra is the existence of a canonical element $\tau\in\pi_{0,-1}(\bS^{0,0})$ and a synthetic $\bE_{\infty}$-ring spectrum $C\tau$ with special properties related to $E_*E$-comodules: 

\begin{theorem} \text{\cite[Section 4.5]{Pst23}}.
\label{tauthm}
The pushout comparison map 
$$\bS^{0,-1}=\Sigma(\nu\bS^{-1})\to \nu(\Sigma\bS^{-1})=\bS^{0,0}$$
defines a canonical element $\tau\in\pi_{0,-1}(\bS^{0,0})$. The cofiber of $\tau$, $C\tau$, is a synthetic $\bE_{\infty}$-ring spectrum and for the synthetic spectra $\nu X,\nu Y\in \Syn_E$, there is a cofiber sequence
    \begin{equation*}
    \begin{tikzcd}
    \Sigma^{0,-1}\nu X \ar[r,"\tau"] & \nu X \ar[r,"i"] & \nu X\otimes C\tau \ar[r,"p"] & \Sigma^{1,-1}\nu X
    \end{tikzcd}
    \end{equation*}
    and a natural isomorphism $$[\nu Y,\nu X\otimes C\tau]_{t,w}\cong \Ext^{w-t,t}_{E_*E}(E_*Y,E_*X)$$ of bigraded abelian groups. In particular, $$\pi_{t,w}(\nu X\otimes C\tau)\cong \Ext^{w-t,t}_{E_*E}(E_*,E_*X)$$ is a regrading of the $\E_2$-page for the $E$-based Adams spectral sequence computing $\pi_*(X^{\wedge}_E)$, where $X^{\wedge}_E$ is the $E$-nilpotent completion of $X$.
\end{theorem}

In fact, Pstr\k{a}gowski proves a stronger result concerning the category $\Mod_{C\tau}(\Syn_E)$ of $C\tau$-modules and Hovey's \cite{Hov04} stable category of $E_*E$-comodules $\Stable_{E_*E}$:

\begin{theorem}\text{\cite[Section 4.5]{Pst23}}
There is a symmetric monoidal embedding $\chi_*:\Mod_{C\tau}(\Syn_E)\hookrightarrow\Stable_{E_*E}$ with the property that $\chi_*(C\tau\otimes \nu X)=E_*X$. The functor $\chi_*$ is an equivalence if $E$ is Landweber exact or if $E=\h\bF_p$.
\end{theorem}

\label{Thm. 2.6}

\begin{remark}
Applying $\pi_{*,*}(-)$ to the cofiber sequence in Theorem 2.5 gives rise to an exact couple

\[
\begin{tikzcd}[column sep=small]
\pi_{*,*}(\nu X) \ar[rr,"\tau"] & & \pi_{*,*}(\nu X) \ar[dl,"i"] \\
& \Ext^{*,*}_{E_*E}(E_*,E_*X) \ar[ul,"p"] &
\end{tikzcd}
\] and hence a $\tau$-Bockstein spectral sequence with $\E_1=\Ext^{*,*}_{E_*E}(E_*,E_*X)[\tau]$ computing $\pi_{*,*}(\nu X_{\tau}^{\wedge})$.
\end{remark}

A $\nu E$-based Adams spectral sequence construction can also be performed in $\Syn_E$ which attempts to compute $\pi_{*,*}(\nu X^{\wedge}_{\nu E})$, the bigraded homotopy groups of the $\nu E$-nilpotent completion of $\nu X$. In \cite[Appendix A]{BHS19}, they construct this spectral sequence, prove convergence properties, and relate it to the $E$-based classical Adams spectral sequences and the $\tau$-Bockstein spectral sequence. We summarize their results as follows:

\begin{theorem} \text{\cite[Appendix A]{BHS19}}.
\label{mainsynthm}
Let $X\in \Sp$ be a spectrum. Then the following statements hold:

\begin{enumerate}     
    \item There is a trigraded $\nu E$-based Adams spectral sequence computing the bigraded homotopy groups $\pi_{*,*}(\nu X^{\wedge}_{\nu E})$ such that the following are true:
    \begin{enumerate}
    \item ${}_{\syn}\E_2^{s,t,w}\cong {}_{\cl}\E_2^{s,t}\otimes \bZ [\tau]$ with differentials $$d_r:{}_{\syn}\E_r^{s,t,w}\to {}_{\syn}\E_r^{s+r,t+r-1,w}$$ where ${}_{\cl}\E_2^{s,t}$ is the classical Adams $\E_2$-page considered in tridegree $(s,t,t)$ and $\tau$ is in tridegree $(0,0,-1)$.
    \item Given a classical differential with $d_{r,\cl}(x)=y$, there is a synthetic differential $d_r(x)=\tau^{r-1}y$. Moreover, every synthetic differential arises this way.
 \end{enumerate}
   \item The following are equivalent:
   \begin{enumerate}
   \item $X$ is $E$-nilpotent complete.
   \item $\nu X$ is $\nu E$-nilpotent complete.
   \item $\nu X$ is $\tau$-complete.
   \end{enumerate} 
   \item Let $X$ be $E$-nilpotent complete. Then the following are equivalent:
   \begin{enumerate}
   \item The $E$-based Adams spectral sequence for $X$ converges strongly.
   \item The $\nu E$-based Adams spectral sequence for $\nu X$ converges strongly.
   \item The $\tau$-Bockstein spectral sequence for $\nu X$ converges strongly.
   \end{enumerate}        
   \item Up to reindexing, the $\nu E$-based Adams spectral sequence for $\nu X$ and the $\tau$-Bockstein spectral sequence for $\nu X$ are isomorphic.
\end{enumerate}
\end{theorem}

\begin{remark}
In \cite{BHS19}, ${}_{\cl}\E_2^{s,t}$ lives in tridegree $(s,t,s)$ and differentials instead have the form $$d_r:{}_{\syn}\E_r^{s,t,w}\to {}_{\syn}\E_r^{s+r,t+r-1,w+1}.$$ We have chosen slightly different indexing so that differentials preserve weight. 
\end{remark}

Theorem~\ref{mainsynthm} implies that for $E=\h\bF_p$ we have an isomorphism of $\bF_p[\tau]$-modules
\begin{equation*}
{}_{\syn}\E_r^{*,*,*}\cong {}_{\cl}\E_r^{*,*}\otimes\bF_p[\tau] \oplus \left(\bigoplus_{k=2}^{r-1}{}_{\cl}\mathrm{B}_k^{*,*}\otimes\bF_p[\tau]/(\tau^{k-1})\right),    
\end{equation*}
for $2\leq r\leq\infty$ where ${}_{\cl}\mathrm{B}_k^{*,*}\subset{}_{\cl}\E_k^{*,*}$ denotes the $k$-boundaries of the ${}_{\cl}\E_k$-page. The $\bF_p[\tau]/(\tau^{k-1})$-summands are how differentials in the Adams spectral sequence are ``remembered'' synthetically. For $p=2$ and $X=tmf_2^{\wedge}$, the $\bF_2[\tau]$-module structure of the $\E_r$-page is too large to work with so, instead, we will work over a larger module, which makes the $\E_r$-page easier to describe. We discuss this more in Section~\ref{AdamsSScalcdetailssection}.

\numberwithin{theorem}{subsection}

\section{Synthetic Adams Spectral Sequence}

\label{AdamsSScalcdetailssection}

In this section, we cover details of the calculation of the $\E_r$-page of the synthetic Adams spectral sequence computing $\pi_{*,*}\nu_{\h\bF_2}tmf$. We imitate the process employed by \cite{BR21}. We let $\cA$ denote the classical Steenrod algebra at $p=2$ and $\cA (n)$ denote the subalgebra of $\cA$ generated by $Sq^1,\ldots,Sq^{2^{n}}$. Throughout this section, we will write ${}_{\mathrm{cl}}\E_r^{*,*}$ for the classical Adams $\E_r$-page of $tmf$, ${}_{\syn}\E_r^{*,*,*}$ for the synthetic Adams $\E_r$-page of $\nu_{\h\bF_2}tmf$, and $tmf$ for $tmf_2^{\wedge}$.

\subsection{\texorpdfstring{$\E_2$-page and $R_i'$-modules}{E2-page and Ri'-modules}}

Classically, the (cohomological) $\E_2$-page of the Adams spectral sequence computing $\pi_*(tmf)$ is 
\[
{}_{\mathrm{cl}}\E_2^{s,t}=\Ext^{s,t}_\cA(H^*(tmf),\bF_2)\cong \Ext^{s,t}_{\cA}(\cA \otimes_{\cA (2)} \bF_2,\bF_2) \cong \Ext^{s,t}_{\cA (2)}(\bF_2,\bF_2)
\]
via a change-of-rings isomorphism and the fact that $\cA$ is free as an $\cA (2)$-module. Here we let $s$ denote homological degree and $t$ denote total degree.

\bigskip

Because $\cA (2)\subset \cA$ is a finite algebra, we can entirely calculate the $\E_2$-page. This was done originally by Shimada-Iwai \cite{SI67}, where they calculate $\Ext^{*,*}_{\cA (2)}(\bF_2,\bF_2)$ as a finitely presented bigraded commutative $\bF_2$-algebra with 13 generators and 54 relations. Following the notation and lexicographic ordering of \cite{BR21}, we write the generators and their Ext bidegree in Table~\ref{E2gentable} (see \cite[Table 3.4]{BR21} for all 54 relations). We let $k=t-s$ denote topological degree or stem.

\numberwithin{table}{section}

\label{E2gentable}
\begin{table}[ht]
\begin{tabular}{cccccccccccccc}
\hline
$k$           & 0     & 1     & 3     & 8     & 12     & 15       & 14    & 17      & 25    & 32  & 20       & 8       & 48    \\
$s$           & 1     & 1     & 1     & 3     & 3     & 3        & 4     & 4       & 5     & 7   & 4        & 4        & 8     \\ \hline
$x\in {}_{\mathrm{cl}}\E_2^{s,k+s}$ & $h_0$ & $h_1$ & $h_2$ & $c_0$ & $\alpha$ & $\beta$ & $d_0$ & $e_0$ & $\gamma$ & $\delta$ & $g$ & $w_1$ & $w_2$ \\ \hline
\end{tabular}
\caption{Algebra generators of $\E_2$-page}
\end{table}

\numberwithin{table}{subsection}

As in every spectral sequence calculation, the next step is to compute differentials. Because $tmf$ is an $\bE_{\infty}$-ring spectrum, this will be a spectral sequence of $\bF_2$-algebras. In particular, the $\E_r$-page for $r\geq 2$ is an $\bF_2$-algebra and differentials $d_r$ satisfy the Leibniz rule. Non-zero differentials together with the presence of these 54 relations in the $\E_2$-page will produce several new algebra generators and make the bookkeeping of this particular calculation complicated to keep track of. To see proofs of differentials in the classical Adams spectral sequence for $tmf$, we refer the reader to \cite[Section 5]{BR21}.

\bigskip

To make the calculation more manageable, \cite{BR21} instead consider ${}_{\mathrm{cl}}\E_r^{*,*}$ as an $R_i:=\bF_2[g,w_1,w_2^{2^i}]$-module, with $0\leq i\leq 2$. In particular, they consider ${}_{\mathrm{cl}}\E_2^{*,*}$ as an $R_0$-module, ${}_{\mathrm{cl}}\E_3^{*,*}$ as an $R_1$-module, and ${}_{\mathrm{cl}}\E_r^{*,*}$ as $R_2$-modules for $r\geq 4$. This is sensible as $g$, $w_1$, and $w_2^4$ will turn out to be permanent cycles and $d_2(w_2)=\alpha\beta g$ and $d_3(w_2^2)=\beta g^4$ are non-zero.

\bigskip

The following results about ${}_{\mathrm{cl}}\E_2^{*,*}$ further justify this choice of using $R_i$-modules:

\begin{proposition}[\cite{SI67}]
\label{E2pagefree}
$\Ext^{*,*}_{\cA (2)}(\bF_2,\bF_2)$ is free as a module over $\bF_2[w_1,w_2]$.    
\end{proposition}

\begin{proposition}[Prop. 3.45, \cite{BR21}]
\label{E2R0mod}
$\Ext^{*,*}_{\cA (2)}(\bF_2,\bF_2)$ is a direct sum of cyclic $R_0$-modules. In particular, we have an isomorphism
\begin{equation*}
\Ext^{*,*}_{\cA (2)}(\bF_2,\bF_2)\cong \bigoplus_{x}\dfrac{\bF_2[g,w_1,w_2]}{\mathrm{Ann}(x)}\{x\}, 
\end{equation*}
where $x$ ranges over all $R_0$-module generators of $\Ext^{*,*}_{\cA (2)}(\bF_2,\bF_2)$. Here $\mathrm{Ann}(x)$ denotes the annihilator ideal of $x$.
\end{proposition}

As \cite{BR21} point out, $\Ext^{*,*}_{\cA (2)}(\bF_2,\bF_2)$ is infinitely generated as an $R_0$-module, but only because of the presence of $h_0$-towers which are manageable. To see a complete list of the cyclic $R_0$-modules generators of $\Ext^{*,*}_{\cA (2)}(\bF_2,\bF_2)$, we refer the reader to \cite[Table 3.6]{BR21}.

\bigskip

Now we return back to the $\h\bF_2$-synthetic setting. As summarized in Theorem~\ref{mainsynthm}, ${}_{\syn}\E_2^{*,*,*}\cong {}_{\mathrm{cl}}\E_2^{*,*}\otimes\bZ[\tau]$ and $d_r^{\syn}(x)=\tau^{r-1}d_r^{\mathrm{cl}}(x)$ for an element $x\in{}_{\syn}\E_r^{*,*,*}$. The synthetic $\E_2$-page is free over $\bZ[\tau]$ and $d_r^{\syn}(\tau)=0$ for all $r\geq 2$, so we define $R_i':=\bF_2[\tau,g,w_1,w_2^{2^i}]=R_i\otimes\bZ[\tau]$. We consider the synthetic $\E_2$-page as an $R_0'$-module, the synthetic $\E_3$-page as an $R_1'$-module, and the synthetic $\E_r$-page as an $R_2'$-module for $r\geq 4$.

\bigskip

As an immediate consequence, we get the synthetic analogs of Proposition~\ref{E2pagefree} and Proposition~\ref{E2R0mod}:

\begin{proposition}
\label{synE2pagefree}
The synthetic $\E_2$-page 
$${}_{\syn}\E_2^{*,*,*}\cong \Ext^{*,*}_{\cA (2)}(\bF_2,\bF_2)[\tau]$$
is free over $\bF_2[\tau,w_1,w_2]$.   
\end{proposition}

\begin{proposition}
\label{synE2R0mod}
The synthetic $\E_2$-page ${}_{\syn}\E_2^{*,*,*}$ is a direct sum of cyclic $R_0'$-modules. In particular, we have an isomorphism
\begin{equation*}
  {}_{\syn}\E_2^{*,*,*}\cong\Ext^{*,*}_{\cA (2)}(\bF_2,\bF_2)[\tau]\cong \bigoplus_{x}\dfrac{\bF_2[\tau,g,w_1,w_2]}{\mathrm{Ann}(x)}\{x\}, 
\end{equation*}
where $x$ ranges over all $R_0'$-module generators of $\Ext^{*,*}_{\cA (2)}(\bF_2,\bF_2)[\tau]$.
\end{proposition}

Analogous to \cite[Table 5.1]{BR21}, in Table~\ref{E2pagetable} we record the $R_0'$-generators $\{x\}$ of ${}_{\syn}\E_2^{*,*,*}$, the annihilator ideal $\mathrm{Ann}(x)$, and $d_2^{\syn}(x)$.

\subsection{\texorpdfstring{$\E_3$-page}{E3-page}}

With ${}_{\syn}\E_2^{*,*,*}$ and $d_2^{\syn}$-differentials at hand, we can compute \linebreak ${}_{\syn}\E_3^{*,*,*}$ as an $R_1'$-module. Many of the details of this are straightforward and are analogous to the details of the calculation of ${}_{\mathrm{cl}}\E_3^{*,*}$ written down in \cite[App. A.1]{BR21}. We will only write down proofs of summands of ${}_{\syn}\E_3^{*,*,*}$ which appear synthetically but not classically; i.e. become $0$ after applying $\tau$-localization. We record the $R_1'$-generators of ${}_{\syn}\E_3^{*,*,*}$ in Table~\ref{E3pagetable}, together with $d_3^{\syn}$-differentials on $R_1'$-module generators of ${}_{\syn}\E_3^{*,*,*}$, and non-cyclic summands in Table~\ref{noncycE3pagetable}.

\bigskip

Let $\Sigma^{s,t,w}R_i'$ denote a shifted copy of $R_i'$ by filtration $s$, total degree $t$, and weight $w$ and $\langle x_1,\ldots,x_n\rangle$ denote the $R_i'$-module generated by $x_1,\ldots,x_n$. As in \cite{BR21}, for cyclic summands we will suppress the shift notation. Altogether, there are three additional cyclic summands, $\langle h_0d_0\rangle$, $\langle h_0^2e_0\rangle$, and $\langle h_0^2e_0w_2\rangle$, and one additional non-cyclic summand, $\langle e_0\gamma,h_0d_0w_2\rangle$, of ${}_{\syn}\E_3^{*,*,*}$ compared to ${}_{\mathrm{cl}}\E_3^{*,*}$: 

\begin{lemma}
\label{h0d0summand}
The element $h_0d_0\in{}_{\syn}\E_2^{5,19,19}$ generates a cyclic summand of ${}_{\syn}\E_3^{*,*,*}$ isomorphic to $R_1'/(g^2,\tau)$.    
\end{lemma}

\begin{proof}
The $d_2^{\syn}$-differentials described in Table~\ref{E2pagetable} determine a chain complex of $R_1'$-modules
\begin{equation*}
\begin{tikzcd}[ampersand replacement=\&]
0 \arrow[r] \& \langle \alpha^3w_2\rangle \arrow[r, "{\begin{pmatrix}\tau g^3w_1\\ \tau w_1\end{pmatrix}}"] \arrow[d,equal] \& \langle\beta\rangle\oplus\langle h_1^2\gamma w_2\rangle \arrow[r, "{\begin{pmatrix}\tau & 0\end{pmatrix}}"] \arrow[d,equal] \& \langle h_0d_0\rangle \arrow[r] \arrow[d,equal] \& 0 \\
\& R_1' \& R_1'\oplus R_1'/(g) \& R_1'/(g^2)
\end{tikzcd}
\end{equation*}
Taking homology, we see that $\langle\alpha^3w_2\rangle$ disappears and $\langle h_0d_0\rangle$, $\langle\beta\rangle\oplus\langle h_1^2\gamma w_2\rangle$ become
\begin{align*}
 \langle h_0d_0\rangle &=R_1'/(g^2,\tau), \\
 \<\beta g^2,h_1^2\gamma w_2\rangle &= \dfrac{\Sigma^{11,66,66}R_1'\oplus\Sigma^{15,90,90}R_1'}{\<(\tau gw_1,\tau w_1),(0,g)\rangle}.
\end{align*}  
\end{proof}

\begin{lemma}
\label{h0h0e0summand}
The element $h_0^2e_0\in{}_{\syn}\E_2^{6,23,23}$ generates a cyclic summand of ${}_{\syn}\E_3^{*,*,*}$ isomorphic to $R_1'/(g,\tau)$.    
\end{lemma}

\begin{proof}
The $d_2^{\syn}$-differentials described in Table~\ref{E2pagetable} determine a chain complex of $R_1'$-modules
\begin{equation*}
\begin{tikzcd}[ampersand replacement=\&]
0 \arrow[r] \& \langle h_2\beta\rangle \arrow[r, "\tau"] \arrow[d,equal] \& \langle h_0^2e_0\rangle \arrow[r] \arrow[d,equal] \& 0 \\
\& R_1'/(g) \& R_1'/(g) \&
\end{tikzcd}
\end{equation*}
Taking homology, we see that $\langle h_2\beta w_2\rangle$ disappears and $\langle h_0^2e_0\rangle$ becomes
\begin{align*}
 \langle h_0^2e_0\rangle = R_1'/(g,\tau).
\end{align*}    
\end{proof}

\begin{lemma}
\label{h0h0e0w2summand}
The element $h_0^2e_0w_2\in{}_{\syn}\E_2^{14,79,79}$ generates a cyclic summand of ${}_{\syn}\E_3^{*,*,*}$ isomorphic to $R_1'/(g,\tau)$.    
\end{lemma}

\begin{proof}
The $d_2^{\syn}$-differentials described in Table~\ref{E2pagetable} determine a chain complex of $R_1'$-modules
\begin{equation*}
\begin{tikzcd}[ampersand replacement=\&]
0 \arrow[r] \& \langle h_2\beta w_2\rangle \arrow[r, "\tau"] \arrow[d,equal] \& \langle h_0^2e_0w_2\rangle \arrow[r] \arrow[d,equal] \& 0 \\
\& R_1'/(g) \& R_1'/(g) \&
\end{tikzcd}
\end{equation*}
Taking homology, we see that $\langle h_2\beta w_2\rangle$ disappears and $\langle h_0^2e_0w_2\rangle$ becomes
\begin{align*}
 \langle h_0^2e_0w_2\rangle = R_1'/(g,\tau).
\end{align*}
\end{proof}

\begin{lemma}
\label{e0gamma_and_h0d0w2summand}
The elements $e_0\gamma\in{}_{\syn}\E_2^{9,51,51}$ and $h_0d_0w_2\in{}_{\syn}\E_2^{13,75,75}$ generate a non-cyclic summand of ${}_{\syn}\E_3^{*,*,*}$ isomorphic to $$\dfrac{\Sigma^{9,51,51}R_1'\oplus\Sigma^{13,75,75}R_1'}{\<(\tau g,\tau),(0,g^2)\rangle}.$$    
\end{lemma}

\begin{proof}
The $d_2^{\syn}$-differentials described in Table~\ref{E2pagetable} determine a chain complex of $R_1'$-modules
\begin{equation*}
\begin{tikzcd}[ampersand replacement=\&]
0 \arrow[r] \& \langle\beta w_2\rangle \arrow[r, "{\begin{pmatrix}\tau g\\ \tau\end{pmatrix}}"] \arrow[d,equal] \& \langle e_0\gamma\rangle\oplus\langle h_0d_0 w_2\rangle \arrow[r] \arrow[d,equal] \& 0 \\
\& R_1' \& R_1'\oplus R_1'/(g^2) \&
\end{tikzcd}
\end{equation*}
Taking homology, we see that $\langle\beta w_2\rangle$ disappears and $\langle e_0\gamma\rangle\oplus\langle h_0d_0 w_2\rangle$ becomes
\begin{align*}
 \langle e_0\gamma, h_0d_0 w_2\rangle &= \dfrac{\Sigma^{9,51,51}R_1'\oplus\Sigma^{13,75,75}R_1'}{\<(\tau g,\tau),(0,g^2)\rangle}.
\end{align*}    
\end{proof}

\subsection{\texorpdfstring{$\E_4$-page}{E4-page}}

With ${}_{\syn}\E_3^{*,*,*}$ and $d_3^{\syn}$-differentials at hand, we can compute \linebreak ${}_{\syn}\E_4^{*,*,*}$ as an $R_2'$-module. Many of the details of this are straightforward and are analogous to the details of the calculation of ${}_{\mathrm{cl}}\E_4^{*,*}$ written down in \cite[App. A.2]{BR21}. Similar to the $\E_3$-page, we will only write down proofs of summands of ${}_{\mathrm{cl}}\E_4^{*,*}$ which either only appear synthetically or are significantly different to the classical case. We record the $R_2'$-generators of ${}_{\syn}\E_4^{*,*,*}$ in Table~\ref{E4pagetable}, together with $d_4^{\syn}$-differentials on $R_2'$-module generators of ${}_{\syn}\E_4^{*,*,*}$, and non-cyclic summands in Table~\ref{noncycE4pagetable}.

\bigskip

Altogether, there are eight additional cyclic summands and two additional non-cyclic summand of ${}_{\syn}\E_4^{*,*,*}$ compared to ${}_{\mathrm{cl}}\E_4^{*,*}$. Six of these cyclic summands are essentially the same as summands of ${}_{\syn}\E_3^{*,*,*}$, gotten by restricting the $R_1'$-action to $R_2'$; namely, these are
\begin{equation*}
\langle h_0d_0\rangle, \langle h_0^2e_0\rangle, \langle h_0^2e_0w_2\rangle,\langle h_0d_0w_2^2\rangle, \langle h_0^2e_0w_2^2\rangle, \langle h_0^2e_0w_2^3\rangle. 
\end{equation*}

The remaining two cyclic summands, $\langle e_0\gamma g+h_0d_0w_2\rangle$ and $\langle (e_0\gamma g+h_0d_0w_2)w_2^2\rangle$, originate from $\langle e_0\gamma,h_0d_0w_2\rangle\subset {}_{\syn}\E_3^{*,*,*}$. The additional non-cyclic summands, $\langle \beta g^2,\gamma^3\rangle$ and $\langle \gamma w_1w_2^2,h_0e_0w_2^3\rangle$, originate from powers of $\tau$ present in relations synthetically. We produce proofs of these four summands and proofs of the summands $\langle 1\rangle$, $\langle w_1w_2^2\rangle$, and $\langle \gamma,h_0e_0w_2,h_1w_2^2\rangle$, whose synthetic descriptions differ significantly from their classical descriptions.

\begin{remark}
    As \cite{BR21} do, in Lemma~\ref{E3w1w2w2summand} we do a change of basis from $h_1^2\gamma w_2$ to $\gamma^3=\beta g^3+h_1^2\gamma w_2$ when going from ${}_{\syn}\E_3^{*,*,*}$ to ${}_{\syn}\E_4^{*,*,*}$.
\end{remark}

\begin{lemma}
\label{1summandE4}
The element $1\in{}_{\syn}\E_3^{0,0,0}$ generates a cyclic summand of ${}_{\syn}\E_4^{*,*,*}$ isomorphic to $$R_2'/(\tau g^4w_1,\tau^2g^6,\tau^2g^2w_1),$$ and the elements $\gamma w_1w_2^2\in{}_{\syn}\E_3^{25,154,154}$, $h_0e_0w_2^3\in{}_{\syn}\E_3^{29,190,190}$ generate a non-cyclic summand of ${}_{\syn}\E_4^{*,*,*}$ isomorphic to
\begin{equation*}
\dfrac{\Sigma^{25,154,154}R_2'\oplus\Sigma^{29,190,190}R_2'}{\langle (\tau^2g^2,0),(\tau g^2,\tau w_1),(0,g)\rangle}.    
\end{equation*}
\end{lemma}

\begin{proof}
The $d_3^{\syn}$-differentials described in Table~\ref{E3pagetable} determine a chain complex of $R_2'$-modules
\begin{equation*}
\resizebox{12.6cm}{!}{
\begin{tikzcd}[ampersand replacement=\&]
0 \arrow[r] \& \langle h_1\gamma w_2\rangle \arrow[r, "{\begin{pmatrix}0\\ \tau^2g^2w_1\\ 0\end{pmatrix}}"] \arrow[d,equal] \& \langle h_1w_2\rangle\oplus\langle \gamma w_2^2,h_0e_0w_2^3\rangle \arrow[r, "{\begin{pmatrix}\tau^2g^2w_1 & \tau^2g^6 & 0\end{pmatrix}}"] \arrow[d,equal] \& \langle 1\rangle \arrow[r] \arrow[d,equal] \& 0 \\
\& R_2'/(g) \& R_2'/(g^2)\oplus \dfrac{R_2'\oplus R_2'}{\<(\tau g^2 w_1,\tau w_1),(0,g)\rangle} \& R_2'/(\tau g^4w_1)
\end{tikzcd}    
}
\end{equation*}
Taking homology, we see that $\langle h_1\gamma w_2\rangle$, $\langle h_1w_2\rangle$ disappear and $\langle 1\rangle$, $\langle \gamma w_2^2,h_0e_0w_2^3\rangle$ become
\begin{align*}
 \langle 1\rangle &=R_2'/(\tau g^4w_1,\tau^2g^6,\tau^2g^2w_1), \\
 \langle \gamma w_1w_2^2,h_0e_0w_2^3\rangle &= \dfrac{\Sigma^{25,154,154}R_2'\oplus\Sigma^{29,190,190}R_2'}{\langle (\tau^2g^2,0),(\tau g^2,\tau w_1),(0,g)\rangle}.
\end{align*}    
\end{proof}

\begin{lemma}
\label{e0plussummand}
The element $e_0\gamma g+h_0d_0w_2\in{}_{\syn}\E_3^{13,75,75}$ generates a cyclic summand of ${}_{\syn}\E_4^{*,*,*}$ isomorphic to $R_2'/(\tau)$.    
\end{lemma}

\begin{proof}
The $d_3^{\syn}$-differentials described in Table~\ref{E3pagetable} determine a chain complex of $R_2'$-modules
\begin{equation*}
\begin{tikzcd}[ampersand replacement=\&]
0 \arrow[r] \& \langle e_0\gamma,h_0d_0w_2\rangle \arrow[r, "{\begin{pmatrix}\tau^2w_1 & 0\end{pmatrix}}"] \arrow[d,equal] \& \langle h_1\delta\rangle \arrow[r] \arrow[d,equal] \& 0 \\
\& \dfrac{R_2'\oplus R_2'}{\<(\tau g,\tau),(0,g^2)\rangle} \& R_2'/(g) \&
\end{tikzcd}
\end{equation*}
Taking homology, we see that $\langle h_1\delta\rangle$ and $\langle e_0\gamma g,h_0d_0w_2\rangle$, after replacing $h_0d_0w_2$ with $e_0\gamma g+h_0d_0w_2$ via change-of-basis, become
\begin{align*}
 \langle h_1\delta\rangle &= R_2'/(g,\tau^2w_1), \\
 \langle e_0\gamma g\rangle &= R_2'/(\tau g^2), \\
 \langle e_0\gamma g+h_0d_0w_2\rangle &= R_2'/(\tau).
\end{align*}    
\end{proof}

\begin{lemma}
\label{E3w1w2w2summand}
The element $w_1w_2^2\in{}_{\syn}\E_3^{20,124,124}$ generates a cyclic summand of ${}_{\syn}\E_4^{*,*,*}$ isomorphic to $R_1'/(\tau g^4,\tau^2g^2)$ and the elements $\beta g^2\in{}_{\syn}\E_3^{11,66,66}$ and $\gamma^3\in{}_{\syn}\E_3^{15,90,90}$ generate a non-cyclic summand of ${}_{\syn}\E_4^{*,*,*}$ isomorphic to
\begin{equation*}
\dfrac{\Sigma^{11,66,66}R_2'\oplus\Sigma^{15,90,90}R_2'}{\langle (g^2,g),(0,\tau w_1),(\tau^2g^2,0)\rangle}.   
\end{equation*}
\end{lemma}

\begin{proof}
The $d_3^{\syn}$-differentials described in Table~\ref{E3pagetable} determine a chain complex of $R_2'$-modules
\begin{equation*}
\begin{tikzcd}[ampersand replacement=\&]
0 \arrow[r] \& \langle h_1w_2^3\rangle \arrow[r, "\tau^2g^2w_1"] \arrow[d,equal] \& \langle w_2^2\rangle \arrow[r, "{\begin{pmatrix}\tau^2g^2 \\ 0\end{pmatrix}}"] \arrow[d,equal] \& \langle \beta g^2,h_1^2\gamma w_2\rangle \arrow[r] \arrow[d,equal] \& 0 \\
\& R_2'/(g^2) \& R_2'/(\tau g^4w_1) \& \dfrac{R_2'\oplus R_2'}{\<(\tau gw_1,\tau w_1),(0,g)\rangle}
\end{tikzcd}
\end{equation*}
Taking homology, we see that $\langle h_1w_2^3\rangle$ disappears and $\langle w_2^2\rangle$ and $\langle\beta g^2,h_1^2\gamma w_2\rangle$, after replacing $h_1^2\gamma w_2$ with $\gamma^3=\beta g^3+h_1^2\gamma w_2$ via change-of-basis, become
\begin{align*}
 \langle w_1w_2^2\rangle &=R_2'/(\tau g^4,\tau^2g^2), \\
 \<\beta g^2\rangle &= R_2'/(\tau^2g^2), \\
 \langle \beta g^2,\gamma^3\rangle &= \dfrac{R_2'\oplus R_2'}{\langle (g^2,g),(0,\tau w_1),(\tau^2g^2,0)\rangle}.
\end{align*}    
\end{proof}

\begin{lemma}
\label{e0plusw2w2summand}
The element $(e_0\gamma g+h_0d_0w_2)w_2^2\in{}_{\syn}\E_3^{29,187,187}$ generates a cyclic summand of ${}_{\syn}\E_4^{*,*,*}$ isomorphic to $R_2'/(\tau)$.    
\end{lemma}

\begin{proof}
The $d_3^{\syn}$-differentials described in Table~\ref{E3pagetable} determine a chain complex of $R_2'$-modules
\begin{equation*}
\begin{tikzcd}[ampersand replacement=\&]
0 \arrow[r] \& \langle e_0\gamma w_2^2,h_0d_0w_2^3\rangle \arrow[r, "{\begin{pmatrix}\tau^2w_1 & 0\end{pmatrix}}"] \arrow[d,equal] \& \langle h_1\delta w_2^2\rangle \arrow[r] \arrow[d,equal] \& 0 \\
\& \dfrac{R_2'\oplus R_2'}{\<(\tau g,\tau),(0,g^2)\rangle} \& R_2'/(g) \&
\end{tikzcd}
\end{equation*}
Taking homology, we see that $\langle h_1\delta w_2^2\rangle$ and $\langle e_0\gamma w_2^2,h_0d_0w_2^3\rangle$, after replacing $h_0d_0w_2^3$ with $(e_0\gamma g+h_0d_0w_2)w_2^2$ via change-of-basis, become
\begin{align*}
 \langle h_1\delta w_2^2\rangle &= R_2'/(g,\tau^2w_1), \\
 \langle e_0\gamma g w_2^2\rangle &= R_2'/(\tau g^2), \\
 \langle (e_0\gamma g+h_0d_0w_2)w_2^2\rangle &= R_2'/(\tau).
\end{align*}     
\end{proof}

\begin{lemma}
\label{3termE4summand}
The elements $\gamma\in{}_{\syn}\E_3^{5,30,30}$, $h_0e_0w_2\in{}_{\syn}\E_3^{13,78,78}$, and $h_1w_2^2\in{}_{\syn}\E_3^{17,114,114}$ generate a non-cyclic summand of ${}_{\syn}\E_4^{*,*,*}$ isomorphic to
\begin{equation*}
    \dfrac{\Sigma^{5,30,30}R_2'\oplus\Sigma^{13,78,78}R_2'\oplus\Sigma^{17,114,114}R_2'}{\langle (\tau g^2w_1,\tau w_1,0),(0,g,0),(\tau^2g^5,0,\tau^2gw_1),(\tau^2g^2w_1,0,0)\rangle}.
\end{equation*}
\end{lemma}

\begin{proof}
The $d_3^{\syn}$-differentials described in Table~\ref{E3pagetable} determine a chain complex of $R_2'$-modules
\begin{equation*}
\begin{tikzcd}[ampersand replacement=\&]
0 \arrow[r] \& \langle h_1\gamma w_2 \rangle\oplus\langle \beta^2w_2^2\rangle \arrow[r, "{\begin{pmatrix}\tau^2g^2w_1 & \tau^2g^5 \\ 0 & 0\\ 0 & \tau^2gw_1\end{pmatrix}}"] \arrow[d,equal] \& \langle\gamma,h_0e_0w_2\rangle\oplus\langle h_1w_2^2\rangle \arrow[r] \arrow[d,equal] \& 0 \\
\& R_2'/(g)\oplus R_2'/(\tau g^2w_1) \& \dfrac{R_2'\oplus R_2'}{\<(\tau g^2 w_1,\tau w_1),(0,g)\rangle}\oplus R_2'/(g^2) \&
\end{tikzcd}
\end{equation*}
Taking homology, we see that $\langle h_1\gamma w_2 \rangle$ disappears and $\langle \beta^2w_2^2\rangle$, $\langle \gamma,h_0e_0w_2\rangle$, and $\langle h_1w_2^2\rangle$ become
\begin{align*}
 \langle \beta^2gw_1w_2^2\rangle &= R_2'/(\tau g), \\
 \langle \gamma,h_0e_0w_2,h_1w_2^2\rangle &= \dfrac{\Sigma^{5,30,30}R_2'\oplus\Sigma^{13,78,78}R_2'\oplus\Sigma^{17,114,114}R_2'}{\langle (\tau g^2w_1,\tau w_1,0),(0,g,0),(\tau^2g^5,0,\tau^2gw_1),(\tau^2g^2w_1,0,0)\rangle}.
\end{align*}    
\end{proof}

\subsection{\texorpdfstring{$\E_{\infty}$-page}{E-infinity page}}

With ${}_{\syn}\E_4^{*,*,*}$ and $d_4^{\syn}$-differentials at hand, we can compute \linebreak ${}_{\syn}\E_5^{*,*,*}$ as an $R_2'$-module. By \cite[Thm. 5.27]{BR21}, ${}_{\syn}\E_5^{*,*,*}={}_{\syn}\E_{\infty}^{*,*,*}$, so we will not need to worry about calculating higher pages. Many of the details of this are straightforward and are analogous to the details of the calculation of ${}_{\mathrm{cl}}\E_{\infty}^{*,*}$ written down in \cite[App. A.3]{BR21}. Similar to the $\E_4$-page, we will only write down proofs of summands of ${}_{\mathrm{cl}}\E_{\infty}^{*,*}$ which either only appear synthetically or are significantly different to the classical case. We record the $R_2'$-generators of ${}_{\syn}\E_{\infty}^{*,*,*}$ in Table~\ref{Einfpagetable}, together with names for homotopy elements detected by the $R_2'$-module summands of ${}_{\syn}\E_\infty^{*,*,*}$, and non-cyclic summands in Table~\ref{noncycEinfpagetable}.

\bigskip

Altogether, there are twelve additional cyclic summands and one additional non-cyclic summand of ${}_{\syn}\E_\infty^{*,*,*}$ compared to ${}_{\mathrm{cl}}\E_{\infty}^{*,*}$. Eight of these cyclic summands have the same description as their analogs in ${}_{\syn}\E_4^{*,*,*}$; namely, these are
\begin{align*}
&\langle h_0d_0\rangle, \langle h_0^2e_0\rangle, \langle e_0\gamma g+h_0d_0w_2\rangle, \langle h_0^2e_0w_2\rangle, \\
\langle h_0d_0&w_2^2\rangle , \langle h_0^2e_0w_2^2\rangle, \langle (e_0\gamma g+h_0d_0w_2)w_2^2\rangle, \langle h_0^2e_0w_2^3\rangle. 
\end{align*}
The remaining four cyclic summands, $\langle\alpha^2g^2\rangle$, $\langle d_0e_0g^2\rangle$, $\langle\alpha^2g^2w_2^2\rangle$, and $\langle d_0e_0g^2w_2^2\rangle$ originate as the kernel of $d_4^{\syn}$-differentials. The additional non-cyclic summand, $\langle \gamma w_1w_2^2,h_0e_0w_2^3\rangle$, originates from powers of $\tau$ present in relations synthetically. We produce proofs of these five summands and proofs of the summands $\langle 1\rangle$, $\langle w_1w_2^2\rangle$, and $\langle \gamma,h_0e_0w_2,h_1w_2^2\rangle$, whose synthetic descriptions differ significantly from their classical descriptions.

\begin{lemma}
\label{1summandEinf}
The element $1\in{}_{\syn}\E_4^{0,0,0}$ generates a cyclic summand of ${}_{\syn}\E_{\infty}^{*,*,*}$ isomorphic to $$R_2'/(\tau g^4w_1,\tau^2g^6,\tau^2g^2w_1,\tau^3gw_1^2).$$
\end{lemma}

\begin{proof}
The $d_4^{\syn}$-differentials described in Table~\ref{E4pagetable} determine a chain complex of $R_2'$-modules
\begin{equation*}
\begin{tikzcd}[ampersand replacement=\&]
0 \arrow[r] \& \langle e_0g\rangle \arrow[r, "\tau^3gw_1^2"] \arrow[d,equal] \& \langle 1\rangle \arrow[r] \arrow[d,equal] \& 0 \\
\& R_2'/(\tau g^2) \& R_2'/(\tau g^4w_1,\tau^2g^6,\tau^2g^2w_1) \&
\end{tikzcd}
\end{equation*}
Taking homology, we see that $\langle 1\rangle$ and $\langle e_0g\rangle$ become
\begin{align*}
\langle 1\rangle &= R_2'/(\tau g^4w_1,\tau^2g^6,\tau^2g^2w_1,\tau^3gw_1^2), \\
\langle e_0g^2\rangle &= R_2'/(\tau g). 
\end{align*}   
\end{proof}

\begin{lemma}
\label{aaggsummand}
The element $\alpha^2g^2\in{}_{\syn}\E_4^{14,78,78}$ generates a cyclic summand of ${}_{\syn}\E_{\infty}^{*,*,*}$ isomorphic to $R_2'/(\tau).$
\end{lemma}

\begin{proof}
The $d_4^{\syn}$-differentials described in Table~\ref{E4pagetable} determine a chain complex of $R_2'$-modules
\begin{equation*}
\begin{tikzcd}[ampersand replacement=\&]
0 \arrow[r] \& \langle \alpha^2g\rangle \arrow[r, "\tau^3w_1^2"] \arrow[d,equal] \& \langle \alpha\beta\rangle \arrow[r] \arrow[d,equal] \& 0 \\
\& R_2'/(\tau g) \& R_2'/(\tau g) \&
\end{tikzcd}
\end{equation*}
Taking homology, we see that $\langle\alpha^2g\rangle$ and $\langle\alpha\beta\rangle$ become
\begin{align*}
\langle\alpha^2g^2\rangle &= R_2'/(\tau), \\
\langle\alpha\beta\rangle &= R_2'/(\tau g,\tau^3w_1^2). 
\end{align*}    
\end{proof}

\begin{lemma}
\label{d0e0ggsummand}
The element $d_0e_0g^2\in{}_{\syn}\E_4^{16,87,87}$ generates a cyclic summand of ${}_{\syn}\E_{\infty}^{*,*,*}$ isomorphic to $R_2'/(\tau).$
\end{lemma}

\begin{proof}
The $d_4^{\syn}$-differentials described in Table~\ref{E4pagetable} determine a chain complex of $R_2'$-modules
\begin{equation*}
\begin{tikzcd}[ampersand replacement=\&]
0 \arrow[r] \& \langle d_0e_0\rangle \arrow[r, "\tau^3w_1^2"] \arrow[d,equal] \& \langle d_0\rangle \arrow[r] \arrow[d,equal] \& 0 \\
\& R_2'/(\tau g^2) \& R_2'/(\tau g^3,\tau^2g^2w_1) \&
\end{tikzcd}
\end{equation*}
Taking homology, we see that $\langle d_0e_0\rangle$ and $\langle d_0\rangle$ become
\begin{align*}
\langle d_0e_0g^2\rangle &= R_2'/(\tau), \\
\langle d_0\rangle &= R_2'/(\tau g^3,\tau^2g^2w_1,\tau^3w_1^2). 
\end{align*}    
\end{proof}

\begin{lemma}
\label{E4w1w2w2summand}
The element $w_1w_2^2\in{}_{\syn}\E_4^{20,124,124}$ generates a cyclic summand of ${}_{\syn}\E_{\infty}^{*,*,*}$ isomorphic to $$R_2'/(\tau g^4,\tau^2g^2,\tau^3gw_1).$$
\end{lemma}

\begin{proof}
The $d_4^{\syn}$-differentials described in Table~\ref{E4pagetable} determine a chain complex of $R_2'$-modules
\begin{equation*}
\begin{tikzcd}[ampersand replacement=\&]
0 \arrow[r] \& \langle e_0gw_2^2\rangle \arrow[r, "\tau^3gw_1"] \arrow[d,equal] \& \langle w_1w_2^2\rangle \arrow[r] \arrow[d,equal] \& 0 \\
\& R_2'/(\tau g^2) \& R_2'/(\tau g^4,\tau^2g^2) \&
\end{tikzcd}
\end{equation*}
Taking homology, we see that $\langle e_0gw_2^2\rangle$ and $\langle w_1w_2^2\rangle$ become
\begin{align*}
\langle e_0g^2w_2^2\rangle &= R_2'/(\tau g), \\
\langle w_1w_2^2\rangle &= R_2'/(\tau g^4,\tau^2g^2,\tau^3gw_1). 
\end{align*}   
\end{proof}

\begin{lemma}
\label{aaggw2w2summand}
The element $\alpha^2g^2w_2^2\in{}_{\syn}\E_4^{30,190,190}$ generates a cyclic summand of ${}_{\syn}\E_{\infty}^{*,*,*}$ isomorphic to $R_2'/(\tau).$
\end{lemma}

\begin{proof}
The $d_4^{\syn}$-differentials described in Table~\ref{E4pagetable} determine a chain complex of $R_2'$-modules
\begin{equation*}
\begin{tikzcd}[ampersand replacement=\&]
0 \arrow[r] \& \langle \alpha^2gw_2^2\rangle \arrow[r, "\tau^3w_1^2"] \arrow[d,equal] \& \langle \alpha\beta w_2^2\rangle \arrow[r] \arrow[d,equal] \& 0 \\
\& R_2'/(\tau g) \& R_2'/(\tau g) \&
\end{tikzcd}
\end{equation*}
Taking homology, we see that $\langle\alpha^2gw_2^2\rangle$ and $\langle\alpha\beta w_2^2\rangle$ become
\begin{align*}
\langle\alpha^2g^2w_2^2\rangle &= R_2'/(\tau), \\
\langle\alpha\beta w_2^2\rangle &= R_2'/(\tau g,\tau^3w_1^2). 
\end{align*}    
\end{proof}

\begin{lemma}
\label{d0e0ggw2w2summand}
The element $d_0e_0g^2w_2^2\in{}_{\syn}\E_4^{32,199,199}$ generates a cyclic summand of ${}_{\syn}\E_{\infty}^{*,*,*}$ isomorphic to $R_2'/(\tau).$
\end{lemma}

\begin{proof}
The $d_4^{\syn}$-differentials described in Table~\ref{E4pagetable} determine a chain complex of $R_2'$-modules
\begin{equation*}
\begin{tikzcd}[ampersand replacement=\&]
0 \arrow[r] \& \langle d_0e_0w_2^2\rangle \arrow[r, "\tau^3w_1^2"] \arrow[d,equal] \& \langle d_0w_2^2\rangle \arrow[r] \arrow[d,equal] \& 0 \\
\& R_2'/(\tau g^2) \& R_2'/(\tau g^3,\tau^2g^2w_1) \&
\end{tikzcd}
\end{equation*}
Taking homology, we see that $\langle d_0e_0w_2^2\rangle$ and $\langle d_0w_2^2\rangle$ become
\begin{align*}
\langle d_0e_0g^2w_2^2\rangle &= R_2'/(\tau), \\
\langle d_0w_2^2\rangle &= R_2'/(\tau g^3,\tau^2g^2w_1,\tau^3w_1^2). 
\end{align*}    
\end{proof}

\begin{lemma}
\label{Einfnoncycsummand1}
The elements $\gamma w_1w_2^2\in{}_{\syn}\E_4^{25,154,154}$ and $h_0e_0w_2^3\in{}_{\syn}\E_4^{29,190,190}$ generate a non-cyclic summand of ${}_{\syn}\E_\infty^{*,*,*}$ isomorphic to
\begin{equation*}
\dfrac{\Sigma^{25,154,154}R_2'\oplus\Sigma^{29,190,190}R_2'}{\langle (0,g),(\tau g^2,\tau w_1),(\tau^2g^2,0),(\tau^3gw_1,0)\rangle}.    
\end{equation*}
\end{lemma}

\begin{proof}
The $d_4^{\syn}$-differentials described in Table~\ref{E4pagetable} determine a chain complex of $R_2'$-modules
\begin{equation*}
\begin{tikzcd}[ampersand replacement=\&]
0 \arrow[r] \& \langle e_0\gamma gw_2^2\rangle \arrow[r, "{\begin{pmatrix}\tau^3gw_1 & 0\end{pmatrix}}"] \arrow[d,equal] \& \langle \gamma w_1w_2^2,h_0e_0w_2^3\rangle \arrow[r] \arrow[d,equal] \& 0 \\
\& R_2'/(\tau g^2) \& \dfrac{R_2'\oplus R_2'}{\langle (0,g),(\tau g^2,\tau w_1),(\tau^2g^2,0)\rangle} \&
\end{tikzcd}
\end{equation*}
Taking homology, we see that $\langle e_0gw_2^2\rangle$ and $\langle\gamma w_1w_2^2,h_0e_0w_2^3\rangle$ become
\begin{align*}
\langle e_0\gamma g^2w_2^2\rangle &= R_2'/(\tau g), \\
\langle\gamma w_1w_2^2,h_0e_0w_2^3\rangle &= \dfrac{\Sigma^{25,154,154}R_2'\oplus \Sigma^{29,190,190}R_2'}{\langle (0,g),(\tau g^2,\tau w_1),(\tau^2g^2,0),(\tau^3gw_1,0)\rangle}. 
\end{align*}    
\end{proof}

\begin{lemma}
\label{Einfnoncycsummand2}
The elements $\gamma\in{}_{\syn}\E_4^{5,30,30}$, $h_0e_0w_2\in{}_{\syn}\E_4^{13,78,78}$, and $h_1w_2^2\in{}_{\syn}\E_4^{17,114,114}$ generate a non-cyclic summand of ${}_{\syn}\E_\infty^{*,*,*}$ isomorphic to
\begin{equation*}
    \dfrac{\Sigma^{5,30,30}R_2'\oplus\Sigma^{13,78,78}R_2'\oplus\Sigma^{17,114,114}R_2'}{\renewcommand{\arraystretch}{1}{\begin{tabular}{c}$\langle (0,g,0),(0,0,g^2),(\tau g^2w_1,\tau w_1,0),$\\$(\tau^2g^5,0,\tau^2gw_1),(\tau^2g^2w_1,0,0),(\tau^3 gw_1^2,0,0)\rangle$\end{tabular}}}.
\end{equation*}
\end{lemma}

\begin{proof}
The $d_4^{\syn}$-differentials described in Table~\ref{E4pagetable} determine a chain complex of $R_2'$-modules
\begin{equation*}
\begin{tikzcd}[ampersand replacement=\&]
0 \arrow[r] \& \langle e_0\gamma g\rangle \arrow[r, "{\begin{pmatrix}\tau^3gw_1^2 & 0\end{pmatrix}}"] \arrow[d,equal] \& \langle \gamma,h_0e_0w_2,h_1w_2^2\rangle \arrow[r] \arrow[d,equal] \& 0 \\
\& R_2'/(\tau g^2) \& \dfrac{R_2'\oplus R_2'\oplus R_2'}{\renewcommand{\arraystretch}{1}{\begin{tabular}{c}$\langle (0,g,0),(0,0,g^2),(\tau g^2w_1,\tau w_1,0),$\\$(\tau^2g^5,0,\tau^2gw_1),(\tau^2g^2w_1,0,0)\rangle$\end{tabular}}} \&
\end{tikzcd}
\end{equation*}
Taking homology, we see that $\langle e_0gw_2^2\rangle$ and $\langle\gamma w_1w_2^2,h_0e_0w_2^3\rangle$ become
\begin{align*}
\langle e_0\gamma g^2\rangle &= R_2'/(\tau g), \\
\langle \gamma,h_0e_0w_2,h_1w_2^2\rangle &= \dfrac{\Sigma^{5,30,30}R_2'\oplus\Sigma^{13,78,78}R_2'\oplus\Sigma^{17,114,114}R_2'}{\renewcommand{\arraystretch}{1}{\begin{tabular}{c}$\langle (0,g,0),(0,0,g^2),(\tau g^2w_1,\tau w_1,0),$\\$(\tau^2g^5,0,\tau^2gw_1),(\tau^2g^2w_1,0,0),(\tau^3 gw_1^2,0,0)\rangle$\end{tabular}}}. 
\end{align*}    
\end{proof}

\section{Synthetic Multiplicative Structure}
\label{multstructsection}

\numberwithin{theorem}{section}

In this section, we study the multiplicative structure of $\pi_{*,*}\nu_{\h\bF_2}tmf$. The work of \cite{BR21} together with Section~\ref{AdamsSScalcdetailssection} give us a complete description of ${}_{\syn}\E_\infty^{*,*,*}$ as an $R_2'$-module. As in the classical calculation, ${}_{\syn}\E_\infty^{*,*,*}$ is infinitely generated but only because of the presence of $h_0$-towers. We record the lists of cyclic and non-cyclic generators, and the names of homotopy elements they detect, in Table~\ref{Einfpagetable} and Table~\ref{noncycEinfpagetable}.

\bigskip

\cite[Ch. 9]{BR21} show that classically, 40 elements generate $\pi_{*}tmf$ as a $\bZ_{2}^{\wedge}$-algebra. One reason for considering $\bZ_2^{\wedge}$ is because $\pi_*tmf$ is a $\pi_0tmf$-algebra and $\pi_0tmf\cong\bZ_2^{\wedge}$. Analogously, we will consider $\pi_{*,*}\nu_{\h\bF_2}tmf$ as a $\pi_{0,*}\nu_{\h\bF_2}tmf$-algebra. The structure of $\pi_{0,*}\nu_{\h\bF_2}tmf$ is more subtle, but understandable using the work of \cite{Pst23} and \cite[App. A]{BHS19}. 

\begin{proposition}
\label{pi0prop}
As a $\bZ_2^{\wedge}[\tau]$-algebra,
\begin{equation*}
    \pi_{0,*}\nu_{\h\bF_2}tmf\cong\bZ_2^{\wedge}[\tau,\ttwo]/(\tau\ttwo=2) 
\end{equation*}
where $\tau\in\pi_{0,-1}$, $2\in\pi_{0,0}$, and $\ttwo\in\pi_{0,1}$.
\end{proposition}

\begin{proof}
To make notation easier, we will let $\pi_{0,w}$ denote $\pi_{0,w}\nu_{\h\bF_2}tmf$ and $\pi_0$ denote $\pi_0tmf$ in this proof.

\bigskip


We first notice that $\pi_{0,w}\cong\pi_0\cong\bZ_2^{\wedge}$ induced by $\tau$-inversion for $w\leq 0$ by \cite[Thm. 4.58]{Pst23}, and the map $$\pi_{0,w}\xrightarrow{\tau}\pi_{0,w-1}$$ is an isomorphism. We still need to understand $w>0$, and this is where we use the $\nu\h\bF_2$-Adams spectral sequence, which for $\nu_{\h\bF_2}tmf$ is a spectral sequence of $\bZ[\tau]$-algebras. Since $tmf:=tmf_2^{\wedge}$ is $2$-complete and finite type, by Theorem~\ref{mainsynthm} $\nu_{\h\bF_2}tmf$ is $\nu\h\bF_2$-complete and the $\nu\h\bF_2$-Adams spectral sequence for $\nu_{\h\bF_2}tmf$ converges strongly. What this specifically means for the $\nu\h\bF_2$-Adams filtration $\F^s\pi_{0,w}$ for a fixed $w>0$ is that
\begin{itemize}
    \item $\lim_s \F^s\pi_{0,w}=0$, $\F^s\pi_{0,w}=\pi_{0,w}$ for $s\leq 0$,
    \item $\pi_{0,w}\cong\lim_s(\pi_{0,w}/\F^s)$,
    \item $\F^s/\F^{s+1}\cong {}_{\syn}\E_{\infty}^{s,s,w}$.
\end{itemize}
In the usual way, we can compute $\pi_{0,w}$ if we can solve the extension problem for the short exact sequences
\begin{equation*}
0\to {}_{\syn}\E^{s,s,w}_{\infty}\to\pi_{0,w}/\F^{s+1}\to\pi_{0,w}/\F^s\to 0.   
\end{equation*}
We solve this inductively using $\tau$-inversion. Now $\E_{\infty}^{s,s,w}=0$ for $s<w$ and for $s=w$, $\E_{\infty}^{s,s,s}=\pi_{0,w}/\F^{w+1}=\bF_2$, detected by $h_0^w$. We write $\ttwo^w$ for the corresponding element in $\pi_{0,w}$, which is sensible since this is a spectral sequence of algebras. The next step is to understand the extension
\begin{equation*}
    0\to\bF_2\{\tau h_0^{w+1}\}\to\pi_{0,w}/\F^{w+2}\to\bF_2\{\ttwo^w\}\to 0.
\end{equation*}

The map ${}_{\syn}\E_{\infty}^{s,s,s}\to {}_{\mathrm{cl}}\E_{\infty}^{s,s}$ sends $h_0^s\mapsto h_0^s$ so that the map $\pi_{0,s}\to\pi_0$ sends $\ttwo^s\mapsto 2^s$. As noted above, we also have $2\mapsto 2$ via the isomorphism $\pi_{0,0}\xrightarrow{\cong}\pi_0$. Synthetically, this means that we must have the relation
\begin{equation*} 
    \tau\ttwo^{w+1}=2\cdot\ttwo^w\in\pi_{0,w},
\end{equation*}
because, if not, the only other option is $2\cdot\ttwo^w=0$ but this contradicts the obvious formula $2\cdot 2^w=2^{w+1}$ after $\tau$-inversion. This solves this extension problem, and applying the same argument inductively we get extensions
\begin{equation*}
    0\to\bF_2\{\tau^{n}h_0^{w+n}\}\to\pi_{0,w}/\F^{w+n+1}\to\bZ/2^n\{\ttwo^w\}\to 0,
\end{equation*}
with $\pi_{0,w}/\F^{w+n+1}\cong\bZ/2^{n+1}\{\ttwo^w\}$ so that
\begin{equation*}
\pi_{0,w}\cong\bZ_2^{\wedge}\{\ttwo^w\}.
\end{equation*}
As a graded object, this gives us a surjection of $\bZ_2^{\wedge}[\tau]$-algebras $\bZ_2^{\wedge}[\tau,\ttwo]\to\pi_{0,*}$. By the discussion above, $\tau\ttwo=2$ holds in $\pi_{0,*}$ so that this map factors through a map of algebras
\begin{equation*}
\bZ_2^{\wedge}[\tau,\ttwo]/(\tau\ttwo=2)\to\pi_{0,*},   
\end{equation*}
which is an isomorphism since both sides are rank one in bidegree $(0,w)$.
\end{proof}

\begin{remark}
In fact, the same argument shows that $$\pi_{0,*}\nu_{\h\bF_2}\bS_2^{\wedge}\cong\pi_{*,*}\nu_{\h\bF_2}\h\bZ_2^{\wedge}\cong\bZ_2^{\wedge}[\tau,\ttwo]/(\tau\ttwo=2)$$
and the commutative diagram
\begin{equation*}
    \begin{tikzcd}
        \nu_{\h\bF_2}\bS_2^{\wedge} \ar[d] \ar[dr] & \\
        \nu_{\h\bF_2}tmf_2^{\wedge} \ar[r] & \nu_{\h\bF_2}\h\bZ_2^{\wedge}
    \end{tikzcd}
\end{equation*}
induces isomorphisms on $\pi_{0,*}$. The relation $\tau\ttwo=2$ first appeared in \cite[Prop. A.20]{BHS19} and the calculation of $\pi_{*,*}\nu_{\h\bF_2}\h\bZ_2^{\wedge}$ is mentioned in a version of \cite{Bur22}, though not proven in this level of detail with the $\nu\h\bF_2$-Adams spectral sequence. Keeping in mind the philosophy that synthetic spectra encapsulate the Adams spectral sequences of spectra, the relation $\tau\ttwo=2=\nu_{\h\bF_2}(2)$ is capturing the fact that $2\in\pi_0tmf$ lives in $\h\bF_2-$Adams filtration 1.  
\end{remark}

We now will begin to describe $\pi_{*,*}\nu_{\h\bF_2}tmf$ as a $\bZ_2^{\wedge}[\tau,\ttwo]/(\tau\ttwo=2)$-algebra. As a first step, we enumerate the $\bF_2[\tau]$-algebra generators of ${}_{\syn}\E_\infty^{*,*,*}$.

\begin{proposition}
\label{synEinfalgelements}
As an $\bF_2[\tau]$-algebra, ${}_{\syn}\E_\infty^{*,*,*}$ is generated by 45 elements. There are two additional generators synthetically: $e_0\gamma g+h_0d_0w_2$ and $(e_0\gamma g+h_0d_0w_2)w_2^2$. We list all 45 elements in Table~\ref{Einfalggentable}, using the notation of \cite{BR21}.
\end{proposition}

\begin{proof}
The proof is essentially the same as the proofs given in \cite[Ch. 5]{BR21}, where algebra generators are determined for ${}_{\cl}\E_3^{*,*,*}$, ${}_{\cl}\E_4^{*,*,*}$, and ${}_{\cl}\E_{\infty}^{*,*,*}$ via inspection of their $R_i$-module structure. We only need to determine if there are additional algebra generators by checking $R_i'$-summands of ${}_{\syn}\E_3^{*,*,*}$, ${}_{\syn}\E_4^{*,*,*}$, and ${}_{\syn}\E_\infty^{*,*,*}$ which only appear synthetically. We determined these additional summands in Section~\ref{AdamsSScalcdetailssection}.

\bigskip

For the additional summands of ${}_{\syn}\E_3^{*,*,*}$, we see that 
\begin{align*}
h_0d_0=h_0&\cdot d_0,  h_0^2e_0=h_0^2\cdot e_0,  h_0^2e_0w_2=h_0e_0\cdot h_0w_2, \\ &e_0\gamma=e_0\cdot\gamma,  h_0d_0w_2=d_0\cdot h_0w_2,    
\end{align*}
so that there are no new algebra generators of ${}_{\syn}\E_3^{*,*,*}$ (see \cite[Table 5.4]{BR21} for full list of classical generators).

\bigskip

For the additional summands of ${}_{\syn}\E_4^{*,*,*}$, we see that
\begin{align*}
h_0d_0&=h_0\cdot d_0,  h_0^2e_0=h_0h_2\cdot d_0,  h_0^2e_0w_2=h_0e_0\cdot h_0w_2, h_0d_0w_2^2=d_0\cdot h_0w_2^2, \\
&\hspace{1.5cm}h_0^2e_0w_2^2=h_0h_2\cdot d_0w_2^2,
h_0^2e_0w_2^3=h_0d_0\cdot h_2w_2^3, \\
&\hspace{2.2cm}e_0\gamma g+h_0d_0w_2=e_0g\cdot\gamma + d_0\cdot h_0w_2, \\ &\hspace{1.6cm}(e_0\gamma g+h_0d_0w_2)w_2^2=e_0gw_2^2\cdot\gamma + d_0\cdot h_0w_2^3, \\ &\hspace{2cm}\gamma w_1w_2^2 = \gamma\cdot w_1w_2^2, h_0e_0w_2^3= d_0\cdot h_2w_2^3,
\end{align*}    

so that there are no new algebra generators of ${}_{\syn}\E_4^{*,*,*}$ (see \cite[Table 5.7]{BR21} for full list of classical generators).

\bigskip

For the additional summands of ${}_{\syn}\E_{\infty}^{*,*,*}$, we see that
\begin{align*}
h_0d_0=h_0\cdot d_0,  &h_0^2e_0=h_0h_2\cdot d_0,  h_0^2e_0w_2=h_0e_0\cdot h_0w_2, h_0d_0w_2^2=d_0\cdot h_0w_2^2, \\
&h_0^2e_0w_2^2=h_0h_2\cdot d_0w_2^2,
h_0^2e_0w_2^3=h_0d_0\cdot h_2w_2^3, \\
&\hspace{.5cm}\gamma w_1w_2^2 = \gamma\cdot w_1w_2^2, h_0e_0w_2^3= d_0\cdot h_2w_2^3, \\
&\hspace{.5cm}\alpha^2g^2=\delta^2 +(\delta')^2, d_0e_0g^2=d_0\cdot\gamma\cdot\delta', \\ \alpha^2g^2&w_2^2=\delta\cdot\delta w_2^2 +\delta'\cdot\delta'w_2^2, d_0e_0g^2w_2^2=d_0w_2^2\cdot\gamma\cdot\delta'.
\end{align*}

The remaining summands, $e_0\gamma g+h_0d_0w_2$ and $(e_0\gamma g+h_0d_0w_2)w_2^2$, through a careful check of the algebra generators (see \cite[Table 5.10]{BR21}) of lesser degree and relations present in ${}_{\syn}\E_{\infty}^{*,*,*}$, are not decomposable and so must be indecomposable algebra generators.
\end{proof}

Returning to $\pi_{*,*}\nu_{\h\bF_2}tmf$, Proposition~\ref{synEinfalgelements} tells us that as a $\bZ_2^{\wedge}[\tau,\ttwo]/(\tau\ttwo=2)$-algebra, $\pi_{*,*}\nu_{\h\bF_2}tmf$ is generated by at most 44 elements. In the homotopy groups of spectra, the presence of relations hidden in higher filtration can cause certain generators to become decomposable. This is the case, for example, with $\alpha^3g+h_0w_1w_2$ and $\alpha^3gw_2^2+h_0w_1w_2^3$, which in homotopy turn out to be the targets of hidden 2-extensions. For formal reasons, synthetically there are no indecomposable algebra generators of ${}_{\syn}\E_{\infty}^{*,*,*}$ which become decomposable in homotopy. We prove this as a consequence of a very general lemma:

\begin{lemma}
\label{taudivisiblelemma}
Suppose we have an Adams-type ring spectrum $E\in\Sp$ and an $E$-nilpotent complete, commutative ring spectrum $R\in\Sp$ such that the $E$-Adams spectral sequence for $R$ converges strongly. Suppose also that there are non-zero, non-$\tau$-divisible homotopy elements $$\alpha\in\pi_{k,k+s}(\nu_E R) \text{ and } \beta\in\pi_{k',k'+s'}(\nu_ER)$$ detected by $a\in{}_{\nu E}\E_{\infty}^{s,k+s,k+s}$ and $b\in{}_{\nu E}\E_{\infty}^{s',k'+s',k'+s'}$. If $ab=0\in{}_{\nu E}\E_{\infty}$, then $\alpha\beta$ is $\tau$-divisible.
\end{lemma}

\begin{proof}
This essentially follows from the characterization of $\nu E$-Adams filtration in \cite[App. A]{BHS19}. The $\nu E$-Adams SS for $\nu_ER$ converges strongly by \cite[Prop. A.16]{BHS19} so that, in particular, $${}_{\nu E}\E_{\infty}^{s+s',k+k'+s+s',k+k'+s+s'}\cong \F^{s+s'}/\F^{s+s'+1}$$ where $\F^*$ is the $\nu E$-Adams filtration on $\pi_{k+k',k+k'+s+s'}$. By assumption, $ab=0$ so that $\alpha\beta\in\F^{s+s'+1}$. By \cite[Cor. A.19]{BHS19}, $\F^{s+s'+1}=\F^{1}_{\tau}$ where $\F^*_{\tau}$ is the $\tau$-Bockstein filtration on $\pi_{k+k',k+k'+s+s'}$. This implies that $\alpha\beta\in\mathrm{im}(\tau)$ by the definition of the $\tau$-Bockstein filtration, so that $\alpha\beta$ is $\tau$-divisible.  
\end{proof}

\begin{corollary}
As a $\bZ_2^{\wedge}[\tau,\ttwo]/(\tau\ttwo=2)$-algebra, $\pi_{*,*}\nu_{\h\bF_2}tmf$ is generated by 44 indecomposable elements. These elements (with the addition of $\ttwo$) are listed in Table~\ref{Einfalggentable} together with the elements of ${}_{\syn}\E_{\infty}^{*,*,*}$ which detect them. 
\end{corollary}

\begin{proof}
Because there are no hidden $\tau$-extensions in the $\nu\h\bF_2$-Adams spectral sequence, every element listed in Table~\ref{Einfalggentable} detects a non-$\tau$-divisible element in \linebreak $\pi_{*,*}\nu_{\h\bF_2}tmf$. If such an indecomposable algebra generator of ${}_{\syn}\E_{\infty}^{*,*,*}$ detects $x\in\pi_{k,k+s}\nu_{\h\bF_2}tmf$ and $x$ is decomposable as $x=yz$ in homotopy, then the product of elements in ${}_{\syn}\E_{\infty}^{*,*,*}$ which detect $y,z$ would need to be $0$ in ${}_{\syn}\E_{\infty}^{*,*,*}$, so that by Lemma~\ref{taudivisiblelemma}, $yz$ is $\tau$-divisible. However this contradicts non-$\tau$-divisibility of $x$.    
\end{proof}

\begin{remark}
    Because synthetically there are four more algebra generators, we have introduced new names for these four elements. The elements $TB_2\in\pi_{56,69}$ and $TB_6\in\pi_{152,181}$ are detected by $\alpha^3g+h_0w_1w_2$ and $(\alpha^3g+h_0w_1w_2)w_2^2$ respectively. We chose these names because after $\tau$-inversion, $TB_2=2\cdot B_2$ and $TB_6=2\cdot B_6$. The elements $x_{62}\in\pi_{62,75}$ and $x_{158}\in\pi_{158,187}$ are detected by $e_0\gamma g+h_0d_0w_2$ and $(e_0\gamma g+h_0d_0w_2)w_2^2$ respectively. We named these elements after the stem they live in.
\end{remark}

We now consider the actions of $B\in\pi_{8,12}\nu_{\h\bF_2}tmf$ and $M\in\pi_{192,224}\nu_{\h\bF_2}tmf$ on $\pi_{*,*}\nu_{\h\bF_2}tmf$. Classically, there are non-zero $B$-torsion classes which arise from differentials killing $w_1$-multiples. Synthetically, $B$-multiples survive for non-$\tau$-divisible classes:

\begin{theorem}
\label{Bnonzerothm}
If $x\in\pi_{*,*}\nu_{\h\bF_2}tmf$ is non-zero and non-$\tau$-divisible, then $x\cdot B^n\in\pi_{*,*}\nu_{\h\bF_2}tmf$ is non-zero for all $n\geq 0$.
\end{theorem}

\begin{proof}
Inspecting the long exact sequence
\begin{equation*}
    \cdots\to \pi_{k,w+1}\nu_{\h\bF_2}tmf\xrightarrow{\tau}\pi_{k,w}\nu_{\h\bF_2}tmf\xrightarrow{i_*}\Ext^{w-k,w}_{\cA}(H^*(tmf),\bF_2)\to\cdots,
\end{equation*}
we see that $i_*(x)\neq 0$ and $i_*(B)=w_1$. By Proposition~\ref{E2pagefree}, $i_*(x)\cdot w_1^n\neq 0$ for $n\geq 0$ and because $i:\nu_{\h\bF_2}tmf\to\nu_{\h\bF_2}tmf\otimes C\tau$ is an $\bE_{\infty}$-ring map, this implies that $x\cdot B^n\neq 0$.
\end{proof}

\begin{remark}
Being non-$\tau$-divisible is key to the validity of the above theorem. For example, $\nu\cdot B^n\in\pi_{3+8n,4+12n}\nu_{\h\bF_2}tmf$ is non-zero for all $n\geq 0$ but $\tau\nu\cdot B^n=0$ for $n\geq 1$. 
\end{remark}

Looking at Figures~\ref{chart1}-\ref{chart8}, it appears as if repeated multiplication by $B$ on $\tau$-power torsion classes preserves the level of $\tau$-power torsion, with the exception of some of the classes of the form $\eta_1^n\kappabar^k$ along a line of slope $\frac{1}{5}$. We show that this is indeed the case:

\begin{theorem}
\label{Btautorsionthm}
Suppose $x\in\pi_{*,*}\nu_{\h\bF_2}tmf$ is a $\tau$-power torsion class such that $\tau^rx=0$ but $\tau^{r-1}x\neq 0$ with $r\geq 1$. Then $\tau^{r-1}x\cdot B^n\neq 0$ for all $n\geq 0$, with the exception of the families of elements
\begin{equation*}
\begin{tabular}{cccc}
$\kappabar^{k+6},$ & $\eta_1\kappabar^{k+6},$ & $\eta_1^2\kappabar^{k+5},$ & $\eta_1^3\kappabar^{k+1},$   
\end{tabular}
\end{equation*}
for $k\geq 0$. For these elements $\tau^2x=0$, $\tau x\neq 0$, and $\tau x\cdot B=0$. 
\end{theorem}

\begin{proof}
In $\pi_{*,*}\nu_{\h\bF_2}tmf$, we only need to consider the cases $r=1,2,3$ since there is no higher $\tau$-power torsion. For $r=1$, the statement holds since ${}_{\syn}\E_{2}^{*,*,*}$ is free as a $\bF_2[\tau,w_1]$-module together with Theorem~\ref{Bnonzerothm}.

\bigskip

For $r=2$, it suffices to investigate the $R_2'$-module generators $y$ of ${}_{\syn}\E_{4}^{*,*,*}$ in Table~\ref{E4pagetable} which have some relations in $\mathrm{Ann}(y)$ containing terms $\tau^2$ and $\tau\cdot w_1^{m+1}$. The only such $R_2'$-modules are
\begin{equation*}
\begin{tabular}{cccccc}
$\langle 1\rangle$, & $\langle \gamma,h_0e_0w_2,h_1w_2^2\rangle,$ & $\langle \beta^2g\rangle$, & $\langle \beta g^2,\gamma^3\rangle$. 
\end{tabular}
\end{equation*}
In considering each module, we only worry about multiples $x\cdot g^kw_1^m$ for $k,m\geq 0$ since multiplication by $w_2^4$ doesn't affect $\tau$-power torsion. For $\langle 1\rangle$, $\mathrm{Ann}(1)=(\tau g^4w_1,\tau^2g^2w_1,\tau^2 g^6)$ so that $g^2w_1^{m+1}$ and $g^3w_1^{m+1}$ are $\tau^2$-torsion and $g^{k+4}w_1^{m+1}$ are $\tau$-torsion for $k,m\geq 0$. In particular, $g^{k+6}w_1$ becomes $\tau$-torsion. For \linebreak $\langle \gamma,h_0e_0w_2,h_1w_2^2\rangle,$ we record the relevant multiples in Table~\ref{3termE4tautortable}.
\numberwithin{table}{section}
\begin{table}[ht]
\caption{$\tau^r$-torsion for $\langle \gamma,h_0e_0w_2,h_1w_2^2\rangle$}
\label{3termE4tautortable}
\begin{tabular}{|c|cccc|}
\hline
$g,w_1$-multiples & $\gamma g^2w_1^{m+1}$ & $(\gamma g^2+h_0e_0w_2)w_1^{m+1}$ & $\gamma g^3w_1^{m+1}$ & $\gamma g^4w_1^{m+1}$     \\ \cline{1-1}
$\tau^r$-torsion  & $\tau^2$              & $\tau$                            & $\tau$                & $\tau$                    \\ \hline
$g,w_1$-multiples & $\gamma g^5w_1^{m+1}$ & $(\gamma g^5+h_1gw_1w_2^2)w_1^m$  & $\gamma g^{n+6}$      & $\gamma g^{n+6}w_1^{m+1}$ \\ \cline{1-1}
$\tau^r$-torsion  & $\tau$                & $\tau^2$                          & $\tau^2$              & $\tau$                    \\ \hline
\end{tabular}
\end{table}
In particular, $\tau$-power torsion is preserved after multiplying by $w_1$ except for $\gamma g^{k+6}$. (Note that there are no hidden $\kappabar$-extensions for the elements in homotopy detected by $h_0e_0w_2$ and $h_1gw_1w_2^2$). For $\langle \beta^2g\rangle$, we first note that there is a relation $\beta^2g\cdot g=\gamma^2g$. Now $\mathrm{Ann}(\beta^2g)=(\tau gw_1,\tau^2g^5)$ so that $\gamma^2g^{k+5}$ is $\tau^2$-torsion and $\gamma^2g^{k+1}w_1^{m+1}$ is $\tau$-torsion. In particular, $\gamma^2g^{k+5}w_1$ becomes $\tau$-torsion. For $\langle \beta g^2,\gamma^3\rangle$, $\gamma^3g^{k+1}$ is $\tau^2$-torsion and $\gamma^3g^{k}w_1^{m+1}$ is $\tau$-torsion. In particular, $\gamma^3g^{k+1}w_1$ becomes $\tau$-torsion.

\bigskip

For $r=3$, it suffices to investigate the $R_2'$-module generators $y$ of ${}_{\syn}\E_{\infty}^{*,*,*}$ in Table~\ref{Einfpagetable} which have some relations in $\mathrm{Ann}(y)$ containing terms $\tau^3$ and at least one of $\tau^2\cdot w_1^{m+1}$ or $\tau\cdot w_1^{m+1}$. The only such $R_2'$-modules are
\begin{equation*}
\resizebox{12.6cm}{!}{%
\begin{tabular}{llllll}
$\langle 1\rangle$, & $\langle d_0\rangle$, & $\langle \gamma,h_0e_0w_2,h_1w_2^2\rangle$, & $\langle d_0\gamma,h_2w_2\rangle$, & $\langle \gamma w_1w_2^2,h_0e_0w_2^3\rangle$, & $\langle d_0\gamma w_2^2,h_2w_2^3\rangle.$
\end{tabular}
}
\end{equation*}

An analysis similar to the one done above for $r=2$ shows that multiplication by $w_1^m$ for $g$-multiples of these classes always preserves $\tau^3$-torsion.
\end{proof}

For studying the action of $M$, we first have the following easy consequence of our calculation of ${}_{\syn}\E_{\infty}^{*,*,*}$:

\begin{theorem}
\label{Mnonzerothm}
If $x\in\pi_{*,*}\nu_{\h\bF_2}tmf$ is non-zero, then $x\cdot M^n\in\pi_{*,*}\nu_{\h\bF_2}tmf$ is non-zero for all $n\geq 0$.    
\end{theorem}

\begin{proof}
This follows immediately from the $R_2'$-module structure of ${}_{\syn}\E_{\infty}^{*,*,*}$, recorded in Table~\ref{Einfpagetable}.
\end{proof}

In fact, more can be said regarding the action of $M$. One of the interesting properties of $tmf$ is that it is polynomial on $M$, generated in some sense by classes between topological degrees 0 and 191 (see \cite[p. 32]{Bau08}, \cite[Thm. 9.27]{BR21}, \cite[Ch. 13]{DFHH14}). A synthetic version of this is also true. Let $N_{[0,192)}\subset\pi_{*,*}\nu_{\h\bF_2}tmf$ denote the $\bZ_2^{\wedge}[\tau,\ttwo,\kappabar,B]/(\tau\ttwo=2)$-submodule generated by all classes in topological degrees $0\leq k <192$. 

\begin{theorem}
\label{Mfreethem}
As $\bZ_2^{\wedge}[\tau,\ttwo,\kappabar,B,M]/(\tau\ttwo=2)$-modules, there is an isomorphism 
\begin{equation*}
    N_{[0,192)}\otimes \bZ_2^{\wedge}[M]\xrightarrow{\cong}\pi_{*,*}\nu_{\h\bF_2}tmf.
\end{equation*}
\end{theorem}

\begin{proof}
By our calculations of ${}_{\syn}\E_{\infty}^{*,*,*}$, $w_2^4$, which detects $M$, acts freely on ${}_{\syn}\E_{\infty}^{*,*,*}$ considered as a $\bF_2[\tau,h_0,g,w_1,w_2^4]$-module. In particular, there is some finitely generated $\bF_2[\tau,h_0,g,w_1]$-module $A$ such that $${}_{\syn}\E_{\infty}^{*,*,*}\cong A\otimes_{\bF_2}\bF_2[w_2^4]$$ as an $\bF_2[\tau,h_0,g,w_1,w_2^4]$-module and the generators of $A$ are concentrated in topological degrees $0\leq k<192$. In addition, the associated graded of $$\bZ_2^{\wedge}[\tau,\ttwo,\kappabar,B]/(\tau\ttwo=2)$$ is isomorphic to $\bF_2[\tau,h_0,g,w_1]$ and $A$ can be chosen so that the associated graded of $N_{[0,192)}$ is isomorphic to $A$. This means that the natural map
\begin{equation*}
    N_{[0,192)}\otimes \bZ_2^{\wedge}[M]\xrightarrow{\varphi}\pi_{*,*}\nu_{\h\bF_2}tmf.
\end{equation*}
induces an isomorphism on associated gradeds, so that $\varphi$ is an isomorphism since the Adams filtration is complete and Hausdorff by strong convergence.
\end{proof}

\begin{remark}
In classical versions of this theorem, the action of $\kappabar$ is not needed since $\kappabar^6=0$. However, synthetically, $\kappabar$ is non-nilpotent and $\kappabar^n$ is not equal to some class of the form $x\cdot M^k$ for $n,k\geq 0$.
\end{remark}

\subsection{Synthetic Hidden Extensions}
\numberwithin{theorem}{subsection}
We now investigate multiplicative structure that is not readily available from ${}_{\syn}\E_{\infty}^{*,*,*}$. We will not attempt to compute the entire multiplicative structure of $\pi_{*,*}\nu_{\h\bF_2}tmf$ as \cite{BR21} do for $\pi_*tmf$. We restrict ourselves to multiplications by $\ttwo$, $\eta$, $\nu$, and $\kappabar$. The $\E_{\infty}$-page ${}_{\syn}\E_{\infty}^{*,*,*}$ computes the associated graded of the $\nu\h\bF_2$-Adams filtration on $\pi_{*,*}\nu_{\h\bF_2}tmf$. When the product of two elements in $\E_{\infty}$ is zero, this means that the corresponding product in homotopy is not necessarily zero, but rather lives in higher filtration. It takes extra work to determine what the actual relation in homotopy is. If the product in homotopy happens to not be zero, then the lowest filtration term of the product is called a \textit{hidden extension}. The precise definition for spectra originates from \cite[Def. 4.2]{Isa19} but we present the synthetic definition in the same generality as \cite[Def. 9.5]{BR21}:

\begin{definition}
\label{hidextdef}
Let $X\otimes Y\to Z$ be a pairing of synthetic spectra in $\Syn_E$, with induced pairings $\pi_{*,*}X\otimes\pi_{*,*}Y\to\pi_{*,*}Z$ and ${}_{\nu E}\E_{\infty}(X)\otimes{}_{\nu E}\E_{\infty}(Y)\to{}_{\nu E}\E_{\infty}(Z)$, where ${}_{\nu E}\E_{\infty}(-)$ denotes the $\E_{\infty}$-page of the $\nu_E E$-Adams spectral sequence. Let $\overline{x}\in \pi_{*,*}X$ be detected by $x\in{}_{\nu E}\E_{\infty}(X)$, and consider classes $y\in{}_{\nu E}\E_{\infty}(Y)$ and $z\in{}_{\nu E}\E_{\infty}(Z)$. We say there is a \textit{hidden $\overline{x}$-extension} from $y$ to $z$ if
\begin{enumerate}
    \item $xy=0$,
    \item there is an element $\overline{y}\in\pi_{*,*}Y$ detected by $y$ such that $\overline{x}\cdot\overline{y}$ is detected by $z$, and
    \item there is no element $\overline{y}'\in\pi_{*,*}Y$ of higher filtration than $\overline{y}$ such that $\overline{x}\cdot\overline{y}'$ is detected by $z$.
\end{enumerate}
\end{definition}

For the remainder of the section, we let $E=\h\bF_2$ and $X=Y=Z=\nu_{\h\bF_2}tmf$.

\bigskip

Suppose $R$ is a 2-complete $\bE_{\infty}$-ring spectrum. There are three general methods for computing synthetic hidden extensions in $\pi_{*,*}(\nu_{\h\bF_2} R)$ which we will use for $R=tmf$:

\begin{enumerate}
    \item The $\bE_{\infty}$-ring structure on $\nu_{\h\bF_2} R$ gives rise to a bigraded commutative ring structure on $\pi_{*,*}(\nu_{\h\bF_2} R)$. Information about the ring structure can reveal certain hidden extensions.
    \item The ring map $\pi_{*,*}(\nu_{\h\bF_2} R)\to\pi_*(R)$ induced by $\tau$-inversion allows for classical hidden extensions to pull back to synthetic ones.
    \item Let $\overline{x}\in\pi_{a,b}(\nu_{\h\bF_2}R)$ and suppose that the image of $\overline{x}$ under the map $i:\pi_{a,b}(\nu_{\h\bF_2}R)\to\pi_{a,b}(\nu_{\h\bF_2}R\otimes C\tau)\cong\Ext^{b-a,b}_{\mathcal{A}}(H^*(R),\bF_2)$ is nonzero. Let $x:=i(\overline{x})$. By Theorem~\ref{tauthm} and the fact that the $\pi_{*,*}(\nu_{\h\bF_2}R)$-module structure of $\pi_{*,*}(\nu_{\h\bF_2} R \otimes C\tau)$ factors through $\pi_{*,*}(\nu_{\h\bF_2}R\otimes C\tau)$, we get a commuting diagram of LES's
\begin{equation*}
\begin{tikzcd}[cramped, mysep=1.2em]
\cdots \ar[r] & \Ext^{w-t-2,w-1}_{\cA}(R) \ar[r] \ar[d,"x"] & \pi_{t,w}(\nu_{\h\bF_2} R) \ar[r,"\tau"] \ar[d,"\overline{x}"] & \pi_{t,w-1}(\nu_{\h\bF_2} R) \ar[r] \ar[d,"\overline{x}"] & \cdots \\
\cdots \ar[r] & \Ext^{(w-t)+(b-a)-2,w+b-1}_{\cA}(R) \ar[r] & \pi_{t+a,w+b}(\nu_{\h\bF_2} R) \ar[r,"\tau"] & \pi_{t+a,w+b-1}(\nu_{\h\bF_2} R) \ar[r] & \cdots \\
\end{tikzcd}
\end{equation*}
where $\Ext^{*,*}_{\cA}(R):=\Ext^{*,*}_{\cA}(H^*(R),\bF_2)$. Via this diagram, some hidden extensions can be determined by the known $\bF_2$-algebra structure of \linebreak $\Ext^{*,*}_{\cA}(H^*(R),\bF_2)$. We refer to this diagram as the $C\tau$-LES. This method was originally used in \cite[Appendix A.2]{BHS19} to deduce a hidden extension in the Toda range of $\pi_{*,*}(\nu_{\h\bF_2}\bS_2^{\wedge})$.
\end{enumerate}

\begin{exmp}
To make notation easier, define $\pi_{a,b}(tmf):=\pi_{a,b}(\nu_{\h\bF_2}tmf)$ and $\Ext^{a,b}(tmf):=\Ext^{a,b}_{\cA}(H^*(tmf),\bF_2)$. Note that $\ttwo\in\pi_{0,1}(tmf)$, $\eta\in\pi_{1,2}(tmf)$, $\nu\in\pi_{3,4}(tmf)$, and $\kappabar\in\pi_{20,24}(tmf)$ are not $\tau$-multiples of any elements and hence have non-zero $C\tau$-images (see \cite[Appendix A.2]{BHS19} for the definitions of these synthetic homotopy elements for $\pi_{*,*}(\bS^{0,0})$. We omit the tilde notation for all homotopy elements except $\ttwo$). These images are given by $h_0\in\Ext^{1,1}$, $h_1\in\Ext^{1,2}$, $h_2\in\Ext^{1,4}$, and $g\in\Ext^{4,24}$ respectively.

\bigskip

Method (3) we find to be the most useful for calculating hidden extensions on $\tau$-torsion elements. For a fully worked-out example of method (3) for $\overline{x}=\ttwo$, see Lemma~\ref{t.17.6}.
\end{exmp}

\begin{remark}
\label{CtauLESremark}
The $C\tau$-LES is remarkable in that it only uses algebraic $\Ext$ data to produce hidden extensions. In fact, it can sometimes be used to calculate \emph{classical} hidden extensions purely from $\Ext$ data. For example, we use the $C\tau$-LES in Lemma~\ref{e.64.14} below to prove the synthetic $\eta$-hidden extension $\eta\eta_1^2\kappa=\tau B \eta_1\epsilon_1$. Division by $\eta_1$ then gives the classical $\eta$-hidden extension $\eta \eta_1\kappa = \tau B \epsilon_1$ (or, in classical terms, $\eta \eta_1\kappa= B\epsilon_1=\epsilon\epsilon_1$). This indicates that the $C\tau$-LES could be yet another tool for calculating hidden extensions in other classical Adams spectral sequence calculations.
\end{remark}

\begin{remark}
In our case, the above three methods suffice for finding nearly all $\ttwo$-, $\eta$-, $\nu$-, and $\kappabar$-hidden extensions. There are three exceptions, namely Lemma~\ref{e.60.14}, Lemma~\ref{e.72.15.1}, and Lemma~\ref{k.120.23} where we use Toda bracket methods to deduce the hidden extensions. We thank Dan Isaksen for suggesting parts of these particular arguments. 
\end{remark}

\begin{remark}
With the exception of two hidden $\kappabar$-extensions, proved in Section~\ref{otherhidextsubsection}, our methods show that all other hidden $\kappabar$-extensions follow from the classical calculations of \cite{BR21} and our calculations of hidden $\ttwo$-, $\eta$-, and $\nu$-extensions.
\end{remark}

We first rule out some particular hidden extensions:

\begin{theorem}
\label{nohidextthm}
There are no hidden $\tau$-, $B$-, or $M$-extensions in $\pi_{*,*}(\nu_{\h\bF_2}tmf)$.   
\end{theorem}

\begin{proof}
No hidden $\tau$-extensions follows from the fact that the $\nu\h\bF_2$-Adams spectral sequence is isomorphic to the $\tau$-Bockstein spectral sequence and the $\tau$-Bockstein spectral sequence has no hidden $\tau$-extensions. No hidden $M$-extensions follows from Theorem~\ref{Mnonzerothm}. For hidden $B$-extensions, there are no hidden $B$-extensions between $\tau$-free classes by \cite[Thm. 9.10]{BR21}. If we have a non-zero class $x\in{}_{\syn}\E_{\infty}^{*,*,*}$ such that $xw_1=0$, then $x$ must be $\tau$-divisible by Theorem~\ref{Bnonzerothm}. Let $x=\tau^ay$ where $y$ is non-$\tau$-divisible and $yw_1\neq 0$ is $\tau^b$-torsion for some $b\leq a$. If $\overline{x}$ and $\overline{y}$ are detected in homotopy by $x$ and $y$ respectively, then $\overline{x}=\tau^a\overline{y}$ and $\overline{y}\cdot B\neq 0$ is $\tau^b$-torsion, since there are no hidden $\tau$-extensions. Then $\overline{x}\cdot B=\tau^a\overline{y}\cdot B=0$ so that there are no hidden $B$-extensions.
\end{proof}

Our main result in this subsection is the following:

\begin{theorem}
\label{mainhidextthm}
All possible hidden $\ttwo$-, $\eta$-, $\nu$-, and $\kappabar$-extensions of $\pi_{*,*}(\nu_{\h\bF_2}tmf)$ are computable and described in Table~\ref{ttwoextensiontable}, Table~\ref{etaextensiontable}, Table~\ref{nuextensiontable}, and Table~\ref{kappabarextensiontable}, respectively.
\end{theorem}

\begin{proof}
In Section 3, we determined ${}_{\syn}\E_{\infty}^{*,*,*}$ as an $R_2'$-module, as recorded in Table~\ref{Einfpagetable} and Table~\ref{noncycEinfpagetable}. Let $\overline{x}$ be one of $\ttwo$, $\eta$, $\nu$, or $\kappabar$. For an $R_2'$-module generator $y\in{}_{\syn}\E_{\infty}^{*,*,*}$, it suffices to determine multiplications $\overline{x}\cdot\overline{y}\kappabar^n$ for $n\geq 0$ by Theorem~\ref{nohidextthm}. These multiplications are proven in Sections~\ref{ttwohidextsubsection}, ~\ref{etahidextsubsection}, ~\ref{nuhidextsubsection}, and ~\ref{otherhidextsubsection}.    
\end{proof}

Because hidden extensions describe the lowest filtration part of a relation associated to the hidden extension, this is only the start in determining the full relation. However, in our case, it turns out that via the classical relations of $\pi_{*}tmf$ determined by \cite{BR21} and Theorem~\ref{mainhidextthm}, this is enough to determine the full synthetic relations associated to these hidden extensions:

\begin{theorem}
\label{fullmultthm}
The bigraded homotopy groups $\pi_{*,*}(\nu_{\h\bF_2}tmf)$, as a module over $\bZ_2^{\wedge}[\tau,\ttwo,\eta,\nu,\kappabar,B,M]/(\tau\ttwo=2)$, are generated by 59 elements, recorded in Table~\ref{homotopymodgentable}, and subject to the relations described by Figures~\ref{chart1}-\ref{chart8}, Table~\ref{kappabarextensiontable}, Table~\ref{otherrelationstable}, Theorem~\ref{Btautorsionthm}, and Theorem~\ref{Mfreethem}.
\end{theorem}

\begin{proof}
All non-hidden extensions and all possible hidden extensions follow from the calculation of ${}_{\syn}\E_{\infty}$, Theorem~\ref{mainhidextthm}, and Theorem~\ref{nohidextthm}. To find the full relations associated to each $\overline{x}$-extension with $\overline{x}\in\{\ttwo,\eta,\nu,\kappabar\}$, we need to look at higher filtration terms. We only need to check $\kappabar$-multiples of generators of ${}_{\syn}\E_{\infty}$ by Theorem~\ref{nohidextthm}.

\bigskip

Possible $\tau$-free terms in higher filtration can be checked classically via \cite[Ch.9]{BR21}. A careful check of Figures~\ref{chart1}-\ref{chart8} together with \cite{BR21} shows that the only such relations are
\begin{align*}
    \nu^2\nu_4&=\tau\eta\epsilon_4+\tau^2\eta_1\kappabar^4, \\
    D_4\kappa&=\ttwo\kappa_4=T\kappa_4+\tau \eta_1^2\kappabar^3, \\
    D_4\kappa\kappabar&=\ttwo\kappa_4\kappabar=T\kappa_4\kappabar+\tau\eta_1^2\kappabar^4, \\
    \eta^2\epsilon_5&=T\kappa_4\kappabar+\tau\eta_1^2\kappabar^4. 
\end{align*}
We prove the relations which only occur synthetically in Section~\ref{otherhidextsubsection}. The other possibility to check for is $\tau$-power torsion terms in higher filtration. However, a careful check of Figures~\ref{chart1}-\ref{chart8} shows there are no such possibilities other than the hidden extensions already proven in Section~\ref{ttwohidextsubsection}, Section~\ref{etahidextsubsection}, Section~\ref{nuhidextsubsection}, and Section~\ref{otherhidextsubsection}. 
\end{proof}

\subsection{Synthetic Hidden Extension Proofs}

\numberwithin{theorem}{subsection}

Suppose we are considering possible hidden $\overline{x}$-extensions from $y$ to $z$. We only need to worry about the following possibilities:
\begin{enumerate}
    \item $y$ and $z$ are both $\tau$-free,
    \item $y$ is $\tau$-free and $z$ is $\tau$-power torsion,
    \item $y$ and $z$ are both $\tau$-power torsion. 
\end{enumerate}

Hidden extensions of the form $(1)$ follow from classical hidden extensions, since $\tau$-localization $\tau^{-1}:\pi_{k,w}\to\pi_k$ sends non-zero $\tau$-free classes to non-zero classes. The possibility where $y$ is $\tau$-power torsion and $z$ is $\tau$-free can't happen since $\tau^{-1}$ sends $\tau$-power torsion classes to zero.

\bigskip

We will use the next four subsections to produce the technical proofs for all synthetic hidden extensions of the form $(2)$ or $(3)$ and other relations associated to certain hidden extensions. We have included classical hidden extensions with appropriate multiples of $\tau$ in the tables in Appendix~\ref{appendix} without proof. We refer the reader to \cite[Section 9.2]{BR21} for the full proofs.

\bigskip

\noindent\textbf{Notation.} For the sake of brevity, we shorten $\Ext^{s,t}_{\cA}(H^*(tmf),\bF_2)$ to $\Ext^{s,t}(tmf)$ and $\pi_{t,w}(\nu_{\h\bF_2}tmf)$ to $\pi_{t,w}(tmf)$. For an element $x\in \E_{\infty}$ of the $\nu \h\bF_2$-based Adams spectral sequence, we write $\{x\}\in\pi_{t,w}$ for a lift of $x$ to an element in the Adams filtration of $\pi_{t,w}$. We use the same notation for synthetic homotopy group elements as classical ones with the exception of $\ttwo\in\pi_{0,1}(\nu_{\h\bF_2}\bS_2^{\wedge})\cong \pi_{0,1}(\nu_{\h\bF_2}tmf)$, which maps to $2\in\pi_0(\bS_2^{\wedge})\cong\pi_0(tmf)$ under $\tau$-localization. This matches the notation used in \cite{Bur20}.

\subsection{\texorpdfstring{$\ttwo$}{Tilde 2}-hidden extensions}
\label{ttwohidextsubsection}

\begin{lemma}
\label{t.17.6}
There is a hidden $\ttwo$-extension from $\ttwo\nu\kappa\in\pi_{17,23}(tmf)$ to $\tau\eta \epsilon B\in\pi_{17,24}(tmf)$.
\end{lemma}

\begin{proof}
The elements $\tau\ttwo\nu\kappa$ and $\tau^2\eta \epsilon B$ are both zero so via the $C\tau$-LES, $\ttwo\nu\kappa$ and $\tau\eta\epsilon B$ lift to the nonzero elements $h_2\beta\in\Ext^{4,22}(tmf)$ and $h_0h_2\beta\in\Ext^{5,23}(tmf)$ respectively. These are related by an $h_0$ multiplication. Hence the $C\tau$-LES commutative diagram
\[
\begin{tikzcd}
\Ext^{4,22}(tmf) \ar[r] \ar[d,"h_0"'] & \pi_{17,23}(tmf) \ar[d,"\ttwo"] \\
\Ext^{5,23}(tmf) \ar[r] & \pi_{17,24}(tmf)
\end{tikzcd}
\]
then tells us that $\ttwo^2\nu\kappa=\tau\eta B\epsilon$.
\end{proof}

\begin{lemma}
\label{t.41.10}
There is a hidden $\ttwo$-extension from $\nu_1\kappa\in\pi_{41,51}(tmf)$ to $\tau \eta\epsilon_1 B\in\pi_{41,52}(tmf)$.
\end{lemma}

\begin{proof}
By looking at the $\E_{\infty}$-page, we see that there is a relation $\ttwo\nu_1=\eta^2\eta_1$. Now there is a classical hidden $\eta$-extension $\eta\eta_1\kappa=\tau\epsilon_1 B$. Putting these two facts together, we see that $\ttwo\nu_1\kappa=\eta^2\eta_1\kappa=\tau\eta\epsilon_1 B$.
\end{proof}

\begin{lemma}
\label{t.47.10}
There is a hidden $\ttwo$-extension from $\nu_1\kappabar\in\pi_{47,57}(tmf)$ to $\tau^2\eta_1\kappa B\in\pi_{47,58}(tmf)$.
\end{lemma}

\begin{proof}
The elements $\tau\nu_1\kappabar$ and $\tau^3\eta_1\kappa B$ are both zero so via the $C\tau$-LES, $\nu_1\kappabar$ and $\tau^2\eta_1\kappa B$ lift to the nonzero elements $w_2\in\Ext^{8,56}(tmf)$ and $h_0w_2\in\Ext^{9,57}(tmf)$ respectively. These are related by an $h_0$ multiplication.
\end{proof}

\begin{lemma}
\label{t.54.12}
There is a hidden $\ttwo$-extension from $\kappa\kappabar^2\in\pi_{54,66}(tmf)$ to $\tau^2\epsilon_1\kappa B\in\pi_{54,67}(tmf)$.
\end{lemma}

\begin{proof}
Note that there is a classical hidden $\ttwo$-extension $\ttwo\kappabar^2=\tau^2\epsilon_1 B$. Hence $\ttwo\kappa\kappabar^2=\tau^2\epsilon_1\kappa B$, which is nonzero.
\end{proof}

\begin{lemma}
\label{t.61.14}
There is a hidden $\ttwo$-extension from $\nu_1\kappa\kappabar\in\pi_{61,75}(tmf)$ to \linebreak $\tau^2\eta_1\kappabar B^2\in\pi_{61,76}(tmf)$.
\end{lemma}

\begin{proof}
The elements $\tau\nu_1\kappa\kappabar$ and $\tau^3\eta_1\kappabar B^2$ are both zero so via the $C\tau$-LES, $\nu_1\kappa\kappabar$ and $\tau^2\eta_1\kappabar B^2$ lift to the nonzero elements $d_0w_2\in\Ext^{8,56}(tmf)$ and $h_0d_0w_2\in\Ext^{9,57}(tmf)$ respectively. These are related by an $h_0$ multiplication.
\end{proof}

\begin{lemma}
\label{t.65.14}
There is a hidden $\ttwo$-extension from $\ttwo\nu_2\kappa\in\pi_{65,79}(tmf)$ to $\tau\eta_1\epsilon_1 B\in\pi_{65,80}(tmf)$.
\end{lemma}

\begin{proof}
The elements $\tau\ttwo\nu_2\kappa$ and $\tau^2\eta_1\epsilon_1B$ are both zero so via the $C\tau$-LES, $\ttwo\nu_2\kappa$ and $\tau\eta_1\epsilon_1 B$ lift to the nonzero elements $h_2\beta w_2\in\Ext^{12,78}(tmf)$ and $h_0h_2\beta w_2\in\Ext^{13,79}(tmf)$ respectively. These are related by an $h_0$ multiplication.
\end{proof}

\begin{lemma}
\label{t.67.14}
There is a hidden $\ttwo$-extension from $\nu_1\kappabar^2\in\pi_{67,81}(tmf)$ to $\tau^2\nu_2B^2\in\pi_{67,82}(tmf)$.
\end{lemma}

\begin{proof}
The elements $\tau\nu_1\kappabar^2$ and $\tau^3\nu_2B^2$ are both zero so via the $C\tau$-LES, $\nu_1\kappabar^2$ and $\tau^2\nu_2B^2$ lift to the elements $gw_2\in\Ext^{12,80}(tmf)$ and $h_0gw_2\in\Ext^{13,81}(tmf)$ respectively. These are related by an $h_0$ multiplication.
\end{proof}

\begin{lemma}
\label{t.68.14}
There is a hidden $\ttwo$-extension from $ D_2\kappabar\in\pi_{68,82}(tmf)$ to $\tau \kappabar^3B\in\pi_{68,83}(tmf)$.
\end{lemma}

\begin{proof}
This follows from a hidden $\nu$-extension which we prove later in Lemma~\ref{n.64.14}. Note that $D_2\kappabar  = \nu\nu_2\kappa$ and $\ttwo\nu\nu_2\kappa = \tau\kappabar^3 B$ via Lemma~\ref{n.64.14}. Hence $\ttwo D_2\kappabar  = \ttwo\nu\nu_2\kappa = \tau D_2\kappabar B$.
\end{proof}

\begin{lemma}
\label{t.113.22}
There is a hidden $\ttwo$-extension from $\ttwo\nu\kappa_4\in\pi_{113,135}(tmf)$ to $\tau\eta\epsilon_4 B\linebreak\in\pi_{113,136}(tmf)$.
\end{lemma}

\begin{proof}
The elements $\tau\ttwo\nu\kappa_4$ and $\tau^2\eta\epsilon_4 B$ are both zero so via the $C\tau$-LES, $\ttwo\nu\kappa_4$ and $\tau\eta\epsilon_4 B$ lift to the nonzero elements $h_2\beta w_2^2 \in\Ext^{20,134}(tmf)$ and $h_0h_2\beta w_2^2\in\Ext^{21,135}(tmf)$ respectively. These are related by an $h_0$ multiplication.
\end{proof}

\begin{lemma}
\label{t.137.26}
There is a hidden $\ttwo$-extension from $\nu_5\kappa\in\pi_{137,163}(tmf)$ to $\tau\eta\epsilon_5 B\linebreak\in\pi_{137,164}(tmf)$.
\end{lemma}

\begin{proof}
This follows from the classical relations $\eta_4\kappa=\eta\kappa_4$ \cite[Table 9.9]{BR21} and $\ttwo\nu_5=\eta\eta_1\eta_4$. We see then that $\ttwo\nu_5\kappa=\eta\eta_1\eta_4\kappa=\eta^2\eta_1\kappa_4=\tau \epsilon_4\eta_1B=\tau\eta\epsilon_5 B$.
\end{proof}

\begin{lemma}
\label{t.143.26}
There is a hidden $\ttwo$-extension from $\nu_5\kappabar\in\pi_{143,169}(tmf)$ to \linebreak $\tau^2\eta_1\kappa_4B\in\pi_{143,170}(tmf)$.
\end{lemma}

\begin{proof}
The elements $\tau\nu_5\kappabar$ and $\tau^3\eta_1\kappa_4B$ are both zero so via the $C\tau$-LES, $\nu_5\kappabar$ and $\tau^2\eta_1\kappa_4B$ lift to the elements $w_2^3\in\Ext^{24,168}(tmf)$ and $h_0w_2^3\in\Ext^{25,169}(tmf)$ respectively. These are related by an $h_0$ multiplication.
\end{proof}

\begin{lemma}
\label{t.144.28}
There is a hidden $\ttwo$-extension from $\kappa\kappa_4\kappabar\in\pi_{144,172}(tmf)$ to $\tau^2\epsilon_5B^2\in\pi_{144,173}(tmf)$.
\end{lemma}

\begin{proof}
This follows from a few non-trivial relations and hidden extensions:
\begin{align*}
\ttwo\kappa\kappa_4\kappabar &= \eta^2\epsilon_1\kappa_4 & & [\ttwo\kappa\kappabar=\eta^2\epsilon_1], \\
&= \eta^2\epsilon_5\kappa & & [\text{\cite{BR21}},\text{Table 9.9}], \\
&= \tau\eta \eta_1\kappa_4 B & & [\text{Lemma }~\ref{e.144.28}], \\
&= \tau^2\epsilon_5B^2 & & [\text{classical } \eta\text{-extension}]. 
\end{align*}

\end{proof}

\begin{lemma}
\label{t.150.28}
There is a hidden $\ttwo$-extension from $\kappa_4\kappabar^2\in\pi_{150,178}(tmf)$ to $\tau^2\epsilon_5\kappa B\in\pi_{150,179}(tmf)$. 
\end{lemma}

\begin{proof}
This uses a classical hidden $\ttwo$-extension and relation:
\begin{align*}
    \ttwo\kappa_4\kappabar^2&= \tau^2\epsilon_1\kappa_4 B & & [\text{classical $\ttwo$-extension}], \\
    &=\tau^2\epsilon_5\kappa B & & [\epsilon_1\kappa_4=\epsilon_5\kappa].
\end{align*}
\end{proof}

\begin{lemma}
\label{t.157.30}
There is a hidden $\ttwo$-extension from $\nu_5\kappa\kappabar\in\pi_{157,187}(tmf)$ to $\tau^2\eta_1B_4\kappabar B\in\pi_{157,188}(tmf)$.
\end{lemma}

\begin{proof}
By Lemma~\ref{t.143.26}, there is a hidden $\ttwo$-extension $\ttwo\cdot\nu_5\kappabar=\tau^2\eta_1\kappa_4 B$. The extension in question follows from the relation $\kappa\kappa_4=B_4\kappabar$. 
\end{proof}

\begin{lemma}
\label{t.161.30}
There is a hidden $\ttwo$-extension from $\ttwo\nu_6\kappa\in\pi_{161,191}(tmf)$ to $\tau\eta_1\epsilon_5 B\in\pi_{161,192}(tmf)$.
\end{lemma}

\begin{proof}
The elements $\tau\ttwo\nu_6\kappa$ and $\tau^2\eta_1\epsilon_5 B$ are both zero so via the $C\tau$-LES, $\ttwo\nu_6\kappa$ and $\tau\eta_1\epsilon_5 B$ lift to the nonzero elements $h_2\beta w_2^3\in\Ext^{28,190}(tmf)$ and $h_0h_2\beta w_2^3\in\Ext^{29,191}(tmf)$ respectively. These are related by an $h_0$ multiplication.
\end{proof}

\begin{lemma}
\label{t.164.30}
There is a hidden $\ttwo$-extension from $D_6\kappabar\in\pi_{164,194}(tmf)$ to $\tau D_4\kappabar^3 B\in\pi_{164,195}(tmf)$.
\end{lemma}

\begin{proof}
This follows from a hidden $\nu$-extension which we prove later in Lemma~\ref{n.161.30}. Note that $D_6\kappabar  = \nu\nu_6\kappa$ and $\ttwo\nu\nu_6\kappa = \tau D_4\kappabar^3 B$ via Lemma~\ref{n.161.30}. Hence $\ttwo D_6\kappabar  = \ttwo\nu\nu_6\kappa = \tau D_4\kappabar^3 B$.
\end{proof}

\subsection{\texorpdfstring{$\eta$}{Eta}-hidden extensions}
\label{etahidextsubsection}

\begin{lemma}
\label{e.15.5}
There is a hidden $\eta$-extension from $\eta\kappa\in\pi_{15,20}(tmf)$ to $\tau \epsilon B\in\pi_{16,22}(tmf)$.
\end{lemma}

\begin{proof}
This follows essentially from the classical $\nu$-extension $\nu\cdot(\nu^2)=\tau\eta\epsilon$. Analogous to the classical case, there is a synthetic relation $\ttwo^2\nu=\eta^3$. Hence $\eta^3\kappa=\ttwo^2\nu\kappa=\nu^3B=\tau\eta\epsilon B$. Multiplication by $\eta$ in these bidegrees is injective so we must have that $\eta^2\kappa=\tau\epsilon B$.  
\end{proof}

\begin{lemma}
\label{e.46.11}
There is a hidden $\eta$-extension from $\epsilon_1\kappa\in\pi_{46,57}(tmf)$ to $\tau \eta_1\kappa B\in\pi_{47,59}(tmf)$.
\end{lemma}

\begin{proof}
This follows from the classical hidden $\eta$-extensions $\eta\eta_1\kappabar=\tau\epsilon_1\kappa$ and $\eta\cdot (\eta\kappabar)=\tau^2\kappa B$. Note that $\tau\eta\epsilon_1\kappa=\eta^2\eta_1\kappabar=\tau^2\eta_1\kappabar B$. Multiplication by $\tau$ is injective in these bidegrees so we must have that $\eta\epsilon_1\kappabar=\tau\eta_1\kappa B$
\end{proof}

\begin{lemma}
\label{e.60.14}
There is no hidden $\eta$-extension from $C_2\in\pi_{60,74}(tmf)$.
\end{lemma}

\begin{proof}
We first note that by \cite[Theorem 9.2]{BR21}, there is a classical Toda bracket $$C_2=\langle \nu, \kappabar,C_1\rangle$$ with no indeterminacy. Multiplying by $\eta$, we shuffle to get
$$\eta \langle \nu, \kappabar,C_1\rangle=\langle \eta,\nu,\kappabar\rangle C_1.$$
The bracket $\langle \eta,\nu,\kappabar\rangle$ can easily be checked by machine and \cite{Moss70} to be $\eta_1$. Hence $\eta C_2=\eta_1 C_1$.

\bigskip

We now show that $\eta_1 C_1=0$. Because classically $C_1$ is detected by $h_0\alpha^3$ and $d_2(\alpha^3)=h_1^2\gamma w_1$, there is a Massey product $\langle h_0,\gamma,h_1^2w_1\rangle$ containing $h_0\alpha^3$ on the $\E_3$-page. There is no indeterminacy and no crossing differentials for this Massey product and so by \cite{Moss70}, $C_1\in\langle 2,\eta_1,\eta^2B\rangle$. Lifting to $\Syn_{\h\bF_2}$ via $\tau^{-1}$, we get a synthetic Toda bracket $\langle\ttwo,\eta_1,\tau\eta^2B\rangle\ni C_1$.

\bigskip

Shuffling with $\eta_1$, we get that
$$\eta_1\langle\ttwo,\eta_1,\tau\eta^2B\rangle=\langle \eta_1,\ttwo,\eta_1\rangle\tau\eta^2B.
$$
If $\langle \eta_1,\ttwo,\eta_1\rangle\subset\pi_{51,61}(tmf)$ contains non-zero elements, the only possibilities are $\ttwo\nu_2$ or $\tau\ttwo^2\nu_2$. Multiplying by $\tau\eta^2B$ in either case will result in 0 since both elements are 0 after multiplying by $\eta$. This completes the proof.
\end{proof}

\begin{lemma}
\label{e.61.14}
There is a hidden $\eta$-extension from $\nu_1\kappa\kappabar\in\pi_{61,75}(tmf)$ to $\tau \kappa\kappabar^2B\in\pi_{62,77}(tmf)$.
\end{lemma}

\begin{proof}
This follows directly from the classical hidden $\eta$-extension $\eta\nu_1=\tau\kappabar B$ by multiplying by $\kappa\kappabar$. 
\end{proof}

\begin{lemma}
\label{e.64.14}
There is a hidden $\eta$-extension from $\eta_1^2\kappa\in\pi_{64,78}(tmf)$ to $\tau \eta_1\epsilon_1 B\in\pi_{65,80}(tmf)$.
\end{lemma}

\begin{proof}
The elements $\tau\eta_1^2\kappa$ and $\tau^2\eta_1\epsilon_1 B$ are both zero so via the $C\tau$-LES, $\eta_1^2\kappa$ and $\tau\eta_1\epsilon_1 B$ lift to the elements $e_0w_2\in\Ext^{12,77}(tmf)$ and $h_1e_0w_2\in\Ext^{13,79}(tmf)$ respectively. These are related by an $h_1$-multiplication.
\end{proof}

\begin{lemma}
\label{e.66.15}
There is a hidden $\eta$-extension from $\epsilon_1\kappa\kappabar\in\pi_{66,81}(tmf)$ to $\tau\eta_1\kappa\kappabar B\linebreak\in\pi_{67,83}(tmf)$.
\end{lemma}

\begin{proof}
This follows directly from the classical hidden $\eta$-extension $\eta\epsilon_1\kappabar=\tau\eta_1\kappa B$.
\end{proof}

\begin{lemma}
\label{e.72.15.1}
There are no hidden $\eta$-extensions from
\[D_3\in\pi_{72,87}(tmf), \qquad
D_7\in\pi_{168,199}(tmf).
\]
\end{lemma}

\begin{proof}
We prove these via a Toda bracket argument. By \cite[Theorem 9.2]{BR21}, there are classical Toda brackets
\[
    D_3 = \langle\nu,\kappabar,2D_2\rangle, \qquad
    D_7 = \langle\nu,\kappabar,2D_6\rangle
\]
with no indeterminacy. These formulas also hold when lifted to $\Syn_{\h\bF_2}$ via $\tau^{-1}$. Multiplying by $\eta$, we shuffle to get
\[
    \eta\langle\nu,\kappabar,\ttwo D_2\rangle = \langle\eta,\nu,\kappabar\rangle \ttwo D_2, \qquad
    \eta\langle\nu,\kappabar,\ttwo D_6\rangle = \langle\eta,\nu,\kappabar\rangle \ttwo D_6.
\]
As in Lemma~\ref{e.60.14}, $\langle\eta,\nu,\kappabar\rangle=\eta_1$. Using the relation $\ttwo\eta_1=0$, we then get that
\[\eta D_3=\ttwo\eta_1D_2=0, \qquad
\eta D_7=\ttwo\eta_1D_6=0.
\]
\end{proof}

\begin{lemma}
\label{e.72.15.2}
There is a hidden $\eta$-extension from $\epsilon_1\kappabar^2\in\pi_{72,87}(tmf)$ to \linebreak $\tau \eta_1\kappabar^2B\in\pi_{73,89}(tmf)$.
\end{lemma}

\begin{proof}
This follows directly from the classical hidden $\eta$-extension $\eta\epsilon_1\kappabar=\tau\eta_1\kappabar B$ by multiplying by $\kappabar$.
\end{proof}

\begin{lemma}
\label{e.111.21}
There is a hidden $\eta$-extension from $\eta\kappa_4\in\pi_{111,132}(tmf)$ to $\tau \epsilon_4B\in\pi_{112,134}(tmf)$.
\end{lemma}

\begin{proof}
This essentially follows from Lemma~\ref{t.113.22}, which says that $\ttwo^2\nu\kappa_4=\tau\eta \epsilon_4B$. By the synthetic relation $\eta^3=\ttwo^2\nu$, we have that $\eta^3\kappa_4=\ttwo^2\nu\kappa_4=\tau\eta \epsilon_4B$. Multiplication by $\eta$ is injective in the bidegrees considered so $\eta^2\kappa_4=\tau \epsilon_4B$.
\end{proof}

\begin{lemma}
\label{e.125.25}
There is a hidden $\eta$-extension from $\eta_4\kappabar B+\eta_1\kappabar^5\in\pi_{125,150}(tmf)$ to $\tau^2\kappa_4B^2\in\pi_{126,152}(tmf)$.
\end{lemma}

\begin{proof}
Note that there is a classical hidden $\eta$-extension $\eta\eta_4\kappabar=\tau^2\kappa_4B$. We also have that $\eta\eta_1\kappabar^4=0$. Hence $\eta(\eta_4\kappabar B+\eta_1\kappabar^5)=\tau^2\kappa_4B^2$.
\end{proof}

\begin{lemma}
\label{e.142.27}
There is a hidden $\eta$-extension from $\epsilon_5\kappa\in\pi_{142,169}(tmf)$ to \linebreak $\tau \eta_1\kappa_4B\in\pi_{143,171}(tmf)$.
\end{lemma}

\begin{proof}
This follows from several non-trivial relations in $\pi_*(tmf)$ \cite[Chapter 9]{BR21}. We have that

\begin{align*}
\tau\eta\epsilon_5\kappa &=  \eta\eta_1\eta_4\kappabar & & \text{\cite[\text{Prop. 9.45}]{BR21}},  \\
&= \tau\eta_4\epsilon_1\kappa & & [\eta\eta_1\kappabar=\tau\epsilon_1\kappa],\\ 
&=\tau\eta\epsilon_1\kappa_4 & & \text{\cite[\text{Table 9.8}]{BR21}},\\
&=\tau\eta_1\epsilon\kappa_4 & & \text{\cite[\text{Prop. 9.40}]{BR21}}, \\
&=\tau^2\eta_1\kappa_4B & & \text{\cite[\text{Prop. 9.40}]{BR21}}.\\
\end{align*}  

Since multiplication by $\tau$ is injective in these bidegrees, we get that $\eta\epsilon_5\kappa=\tau \eta_1\kappa_4B$.
\end{proof}

\begin{lemma}
\label{e.143.26}
There is a hidden $\eta$-extension from $\nu_5\kappabar\in\pi_{143,169}(tmf)$ to \linebreak $\tau\kappa\kappa_4\kappabar\in\pi_{144,171}(tmf)$.
\end{lemma}

\begin{proof}
This follows directly from the classical hidden $\eta$-extension $\eta\nu_5=\tau\kappa\kappa_4$ by multiplying by $\kappabar$.
\end{proof}

\begin{lemma}
\label{e.144.28}
There is a hidden $\eta$-extension from $\kappa\kappa_4\kappabar\in\pi_{144,172}(tmf)$ to $\tau \nu_5\kappa B\in\pi_{145,174}(tmf)$.
\end{lemma}

\begin{proof}
This follows directly from the classical hidden $\eta$-extension $\eta\kappa_4\kappabar=\tau \nu_5B$ by multiplying by $\kappa$.
\end{proof}

\begin{lemma}
\label{e.156.30}
There is no hidden $\eta$-extension from $C_6\in\pi_{156,186}(tmf)$.
\end{lemma}

\begin{proof}
Note that by \cite[Theorem 9.48]{BR21}, $B_4C_2=C_6B$ (this exact relation also holds synthetically). By Lemma~\ref{e.60.14}, $\eta C_6B=\eta B_4C_2=0$. Since multiplication by $B$ is injective in the bidegrees we are looking at, we must also have that $\eta C_6=0$.
\end{proof}

\begin{lemma}
\label{e.160.30}
There is a hidden $\eta$-extension from $\eta_1^2\kappa_4\in\pi_{160,190}(tmf)$ to $\tau \eta_1(B_5+\epsilon_5)B\in\pi_{161,192}(tmf)$.
\end{lemma}

\begin{proof}
Note that $\tau\eta_1^2\kappa_4$ and $\tau^2 \eta_1(B_5+\epsilon_5)B$ are both zero so via the $C\tau$-LES, $\eta_1^2\kappa_4$ and $\tau \eta_1(B_5+\epsilon_5)B$ lift to the elements $e_0w_2^3\in\Ext^{28,189}(tmf)$ and $h_1e_0w_2^3\in\Ext^{29,191}(tmf)$ respectively. These are related by an $h_1$-multiplication.
\end{proof}

\begin{lemma}
\label{e.162.31}
There is a hidden $\eta$-extension from $\epsilon_5\kappa\kappabar\in\pi_{162,193}(tmf)$ to $\tau\nu_6B^2\in\pi_{163,195}(tmf)$.
\end{lemma}

\begin{proof}
Note that there is a relation $\nu_6B^2=\eta_1\kappa_4\kappabar B$. Hence by Lemma~\ref{e.144.28}, $\eta\epsilon_5\kappa\kappabar=\tau\eta_1\kappa_4\kappabar B=\tau\nu_6B^2$.
\end{proof}

\begin{lemma}
\label{e.163.30}
There is a hidden $\eta$-extension from $\nu_5\kappabar^2\in\pi_{163,193}(tmf)$ to $\tau \kappa\kappa_4\kappabar^2\in\pi_{164,195}(tmf)$.
\end{lemma}

\begin{proof}
This follows directly from the classical hidden $\eta$-extension $\eta\nu_5=\tau\kappa\kappa_4$ by multiplying by $\kappabar^2$.
\end{proof}

\subsection{\texorpdfstring{$\nu$}{nu}-hidden extensions}
\label{nuhidextsubsection}

\begin{lemma}
\label{n.46.11}
There is a hidden $\nu$-extension from $\epsilon_1\kappa\in\pi_{46,57}(tmf)$ to $\tau^2\nu_1\kappa B\in\pi_{49,61}(tmf)$.
\end{lemma}

\begin{proof}
Note that there is a classical hidden $\nu$-extension $\nu\epsilon_1=\tau^2\nu_1B$. Hence $\nu\epsilon_1\kappa=\tau^2\nu_1\kappa B$.
\end{proof}

\begin{lemma}
\label{n.64.14}
There is a hidden $\nu$-extension from $\eta_1^2\kappa\in\pi_{64,78}(tmf)$ to $\tau^2\nu_2B^2\in\pi_{67,82}(tmf)$.
\end{lemma}

\begin{proof}
The elements $\tau\eta_1^2\kappa$ and $\tau^3\nu_2B^2$ are both zero and so via the $C\tau$-LES, $\eta_1^2\kappa$ and  $\tau^2\nu_2B^2$ lift to the nonzero elements $e_0w_2\in\Ext^{12,77}(tmf)$ and $h_2e_0w_2\in\Ext^{13,81}(tmf)$ respectively. These are related by an $h_2$ multiplication.
\end{proof}

\begin{lemma}
\label{n.65.14}
There is a hidden $\nu$-extension from $\ttwo\nu_2\kappa\in\pi_{65,79}(tmf)$ to $\tau\kappa^3B\in\pi_{68,83}(tmf)$.
\end{lemma}

\begin{proof}
The elements $\tau\ttwo\nu_2\kappa$ and $\tau^2\kappa^3B$ are both zero and so via the $C\tau$-LES, $\ttwo\nu_2\kappa$ and $\tau\kappa^3B$ lift to the nonzero elements $h_2\beta w_2\in\Ext^{12,78}(tmf)$ and $h_2^2\beta w_2\in\Ext^{13,82}(tmf)$ respectively. These are related by an $h_2$ multiplication.
\end{proof}

\begin{lemma}
\label{n.71.16}
There is a hidden $\nu$-extension from $\eta_1\epsilon_1\kappa\in\pi_{71,87}(tmf)$ to \linebreak $\tau^2\epsilon_1\kappa\kappabar B\in\pi_{74,91}(tmf)$.
\end{lemma}

\begin{proof}
This follows directly from the classical hidden $\nu$-extension $\nu\eta_1=\tau^2\kappabar B$ by multiplying by $\epsilon_1\kappa$.
\end{proof}

\begin{lemma}
\label{n.129.25}
There is a hidden $\nu$-extension from $\eta_1B_4\in\pi_{129,154}(tmf)$ to \linebreak $\tau^2\kappa\kappa_4B\in\pi_{132,158}(tmf)$.
\end{lemma}

\begin{proof}
Note that $\nu B_4\in\pi_{107,128}(tmf)$ is non-zero but $\tau\nu B_4=0$. We also have that $\tau^3\kappa\kappa_4B=0$ so that $\nu B_4$ and $\tau^2\kappa\kappa_4 B$ lift via the $C\tau$-LES to $\alpha w_2^2\in\Ext^{19,127}(tmf)$ and $\alpha\gamma w_2^2\in\Ext^{24,157}(tmf)$ respectively. These are related by a $\gamma$ multiplication. Since $\gamma$ is a permanent cycle which detects $\eta_1$, this then implies that $\nu\eta_1 B_4=\tau^2\kappa\kappa_4B$. 
\end{proof}

\begin{lemma}
\label{n.142.27}
There is a hidden $\nu$-extension from $\epsilon_5\kappa\in\pi_{142,169}(tmf)$ to \linebreak $\tau^2\nu_5\kappa B\in\pi_{145,173}(tmf)$.
\end{lemma}

\begin{proof}
This follows directly from the classical hidden $\nu$-extension $\nu\epsilon_5=\tau^2\nu_5\epsilon=\tau^2\nu_5B$.
\end{proof}

\begin{lemma}
\label{n.154.30}
There is a hidden $\nu$-extension from $\eta_1^2B_4\in\pi_{154,184}(tmf)$ to \linebreak $\tau^2\eta_1\kappa\kappa_4B\in\pi_{157,188}(tmf)$.
\end{lemma}

\begin{proof}
This follows directly from Lemma~\ref{n.129.25} by multiplying by $\eta_1$.
\end{proof}

\begin{lemma}
\label{n.160.30}
There is a hidden $\nu$-extension from $\eta_1^2\kappa_4\in\pi_{160,190}(tmf)$ to \linebreak $\tau^2\eta_1\kappa_4\kappabar B\in\pi_{163,194}(tmf)$.
\end{lemma}

\begin{proof}
This follows directly from the classical hidden $\nu$-extension $\nu\eta_1=\tau^2\kappabar B$ by multiplying by $\eta_1\kappa_4$.
\end{proof}

\begin{lemma}
\label{n.161.30}
There is a hidden $\nu$-extension from $\ttwo\nu_6\kappa\in\pi_{161,191}(tmf)$ to \linebreak $\tau D_4\kappabar^3B\in\pi_{164,195}(tmf)$.
\end{lemma}

\begin{proof}
The elements $\tau\ttwo\nu_6\kappa$ and $\tau^2 D_4\kappabar^3B$ are both zero and so via the $C\tau$-LES, $\ttwo\nu_6\kappa$ and $\tau D_4\kappabar^3B$ lift to the nonzero elements $h_2\beta w_2^3\in\Ext^{28,190}(tmf)$ and $h_2^2\beta w_2^3\in\Ext^{29,194}(tmf)$ respectively. These are related by an $h_2$ multiplication.
\end{proof}

\begin{lemma}
\label{n.167.32}
There is a hidden $\nu$-extension from $\eta_1\epsilon_5\kappa\in\pi_{167,199}(tmf)$ to $\tau^2\epsilon_5\kappa\kappabar B\in\pi_{170,205}(tmf)$.
\end{lemma}

\begin{proof}
This follows from the classical hidden $\nu$-extension $\nu\eta_1\epsilon_5=\tau^2\epsilon_5\kappabar B$.    
\end{proof}

\begin{lemma}
\label{n.172.33}
There is a hidden $\nu$-extension from $TB_6\kappabar\in\pi_{172,205}(tmf)$ to $\tau\eta_1^7\in\pi_{173,209}(tmf)$.
\end{lemma}

\begin{proof}
The elements $\tau TB_6\kappabar$ and $\tau^2\eta_1^7$ are both zero and so via the $C\tau$-LES, $TB_6\kappabar$ and $\tau\eta_1^7$ lift to the nonzero elements $\beta d_0w_2^3\in\Ext^{31,204}(tmf)$ and $h_2\beta d_0w_2^3\in\Ext^{32,208}(tmf)$ respectively. These are related by an $h_2$ multiplication.
\end{proof}

\subsection{Other extensions and relations}
\label{otherhidextsubsection}

In this subsection, we collect proofs of miscellaneous extensions and relations. In particular, we prove two synthetic hidden $\kappabar$-extensions and three synthetic relations.

\begin{lemma}
\label{k.120.23}
There is no hidden $\kappabar$-extension from $D_5\in\pi_{120,143}(tmf)$.
\end{lemma}

\begin{proof}
We use a Toda bracket argument to prove this. By \cite[Thm. 9.2]{BR21}, there is a classical Toda bracket
\begin{equation*}
    D_5=\langle \nu,\kappabar,4D_4\rangle
\end{equation*}
with no indeterminacy. Lifting up to $\Syn_{\h\bF_2}$ via $\tau^{-1}$, we get the Toda bracket
\begin{equation*}
    D_5=\langle \nu,\kappabar,\ttwo^2D_4\rangle.
\end{equation*}
Multiplying by $\kappabar$, we shuffle to get
\begin{equation*}
    D_5\cdot\kappabar=\langle \nu,\kappabar,\ttwo^2D_4\rangle\kappabar =\nu\langle\kappabar,\ttwo^2D_4,\kappabar\rangle.
\end{equation*}
By Adams chart inspection, $\pi_{137,163}(tmf)$ contains no $\tau$-power torsion. This means that the Toda bracket $\langle\kappabar,\ttwo^2D_4,\kappabar\rangle\subset\pi_{137,163}(tmf)$ can only contain elements of the form
\begin{equation*}
    a\nu_5\kappa +\tau B\cdot x
\end{equation*}
where $x\in\pi_{129,152}$ and $a\in\{0,1\}$. Since there are relations $\nu\cdot\nu_5=0$ and $\nu\cdot\tau B=0$, this must mean that $D_5\kappabar=0$.
\end{proof}

\begin{lemma}
\label{k.130.25}
There is a hidden $\kappabar$-extension from $T\kappa_4\kappabar=\ttwo\kappa_4\kappabar+\tau\eta_1^2\kappabar^4\in\pi_{130,155}(tmf)$ to $\tau\eta_1^2\kappabar^5\in\pi_{150,179}(tmf)$.
\end{lemma}

\begin{proof}
For the definition of $T\kappa_4$, see the proof of Lemma~\ref{r.110.21} below or Table~\ref{Einfpagetable}. The elements $\tau T\kappa_4\kappabar$ and $\tau^2\eta_1^2\kappabar^5$ are both zero and so via the $C\tau$-LES, $T\kappa_4\kappabar$ and $\tau\eta_1^2\kappabar^5$ lift to the nonzero elements $\beta gw_2^2 \in\Ext^{23,154}(tmf)$ and $\beta g^2w_2^2\in\Ext^{27,178}(tmf)$ respectively. These are related by a $g$ multiplication.    
\end{proof}

\begin{lemma}
\label{r.110.21}
There is a relation $\ttwo\cdot\kappa_4=D_4\kappa=T\kappa_4+\tau\eta_1^2\kappabar^3$ in $\pi_{110,131}(tmf)$.
\end{lemma}

\begin{proof}
Note that $T\kappa_4$ is by definition the $\tau$-torsion homotopy class in $\pi_{110,131}(tmf)$ in filtration 21 detected by $h_0d_0w_2^2$. This then follows from the classical hidden $2$- and $\kappa$-extension $2\cdot \kappa_4=\eta_1^2\kappabar^3=D_4\cdot\kappa$ \cite[Prop. 9.46]{BR21}. Synthetically, $\ttwo \kappa_4$ and $D_4\kappa$ are also detected by $h_0d_0w_2^2$ in ${}_{\syn}\E_{\infty}$ so that this isn't a hidden extension in the sense of Definition~\ref{hidextdef}.
\end{proof}

\begin{lemma}
\label{r.130.25}
There are relations $\ttwo\cdot\kappa_4\kappabar=D_4\kappa\kappabar=T\kappa_4\kappabar+\tau\eta_1^2\kappabar^4$ and $\eta\cdot\eta\epsilon_5=D_4\kappabar^2=T\kappa_4\kappabar+\tau\eta_1^2\kappabar^4$ in $\pi_{130,155}(tmf)$.    
\end{lemma}

\begin{proof}
For the first relation, the proof is the same as the proof of Lemma~\ref{r.110.21} after multiplying by $\kappabar$. The second relation follows from the relation $h_0d_0gw_2^2=h_1^2\delta'w_2^2$ in ${}_{\syn}\E_{\infty}$ and the classical hidden $\eta$-extension $\eta\cdot\eta\epsilon_5=\eta_1^2\kappabar^4$.  
\end{proof}

\newpage

\appendix

\FloatBarrier

\numberwithin{theorem}{subsection}

\section{Charts and Tables}

\maketitle

\numberwithin{table}{section}

\label{appendix}

In this section, we present Adams charts of $\pi_{k,w}\nu_{\h\bF_2}tmf$ for $0\leq k\leq 192$ and tables of ${}_{\syn}\E_{r}^{*,*,*}$ generators for $2\leq r\leq\infty$, non-cyclic modules of ${}_{\syn}\E_{r}^{*,*,*}$, algebra generators of ${}_{\syn}\E_{\infty}^{*,*,*}$ and $\pi_{*,*}\nu_{\h\bF_2}tmf$, module generators of $\pi_{*,*}\nu_{\h\bF_2}tmf$, all $\kappabar$-multiplications on elements detected by generators of ${}_{\syn}\E_{\infty}^{*,*,*}$, all $\ttwo$-, $\eta$-, and $\nu$-hidden extensions, and a table of some miscellaneous relations. The charts were produced using Hood Chatham's \texttt{SPECTRALSEQUENCES} LaTeX package.

\bigskip

For the charts, we use the following conventions:
\begin{itemize}
\item Black dots indicate copies of $\bF_2[\tau]$.
\item Red dots indicate copies of $\bF_2[\tau]/\tau\cong \bF_2$.
\item Blue dots indicate copies of $\bF_2[\tau]/\tau^2$.
\item Green dots indicate copies of $\bF_2[\tau]/\tau^3$.
\item Orange solid lines indicate $\ttwo$- and $\eta$-hidden extensions.
\item Orange dashed lines indicate $\nu$-hidden extensions.
\end{itemize}

The $x$-axis corresponds to the stem $k$ of $\pi_{k,w}(tmf)$ and the $y$-axis corresponds to the filtration of the $\nu\h\bF_2$-Adams spectral sequence. The curvature of a hidden extension has no mathematical meaning and is solely there for aesthetic reasons. The reader should also be wary that the target of a hidden extension in the chart is never the element itself but rather some $\tau$-multiple of that element. The vertical length of the extension determines which $\tau$-multiple it is.

\bigskip

We have included labels for several homotopy elements, particularly those which are sources of a hidden extension.

\bigskip

For the tables, we have included entries corresponding to

\begin{itemize}
    \item $k$: stem
    \item $w$: synthetic weight
    \item $s$: filtration
    \item $t$: internal degree
    \item $x$: name of element in $\Ext$ that survives to $\E_{\infty}$
    \item $\mathrm{Ann}(x)$: the $R_i'$-annihilator ideal of $\langle x\rangle$
    \item $d_r(x)$: synthetic $d_r$ differential on $x$, via \cite{BR21}
    \item $y$: name of homotopy element in $\pi_{k,w}$, following the notation of \cite{BR21}
    \item $\alpha$-extension: a hidden $\alpha$-extension calculation, following the notation of \cite{BR21}. If an entry is blank, then either there is a non-zero $\alpha$-multiplication detected in $\E_{\infty}$ or for degree reasons there is no possible hidden extension.
    \item Proof: where the hidden extension is proven. For proofs of classical hidden extensions, see \cite[Section 9.2]{BR21}.
\end{itemize}

\newpage

\begin{minipage}{.96\linewidth}

\includegraphics[width=\linewidth]{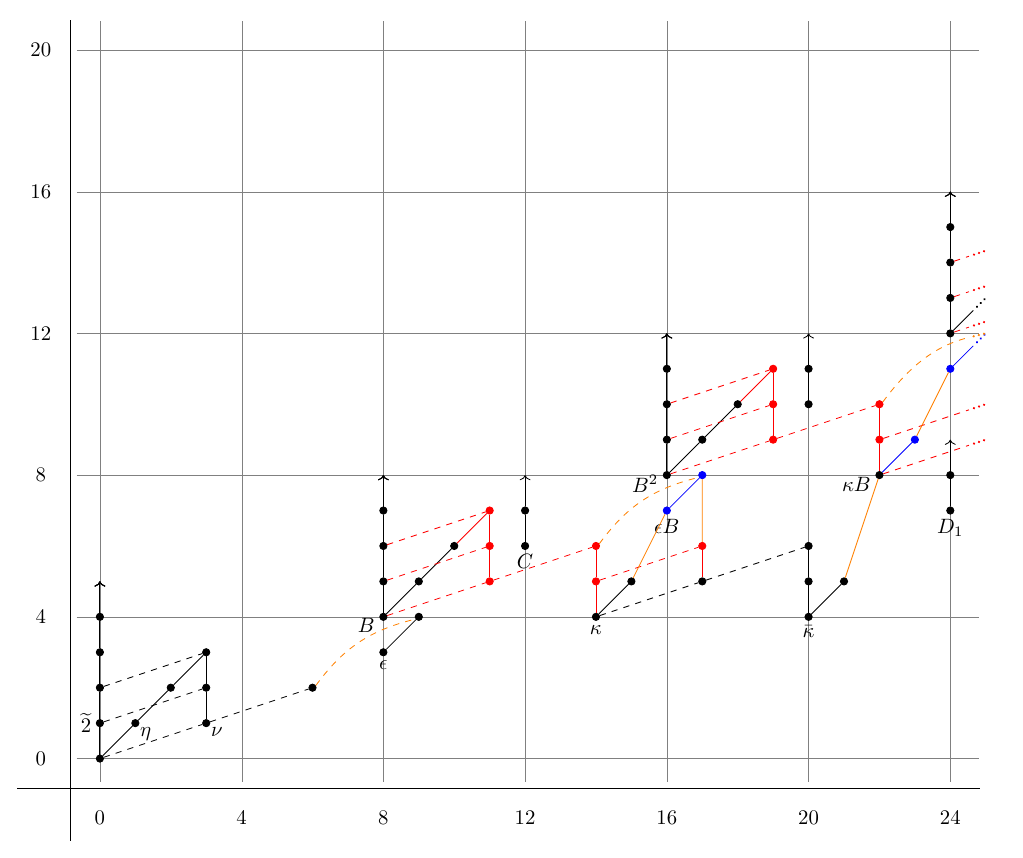}
\captionof{figure}{$\pi_{k,w}(tmf)$, $0\leq k\leq 24$}
\label{chart1}
\vspace{1cm}
\includegraphics[width=\linewidth]{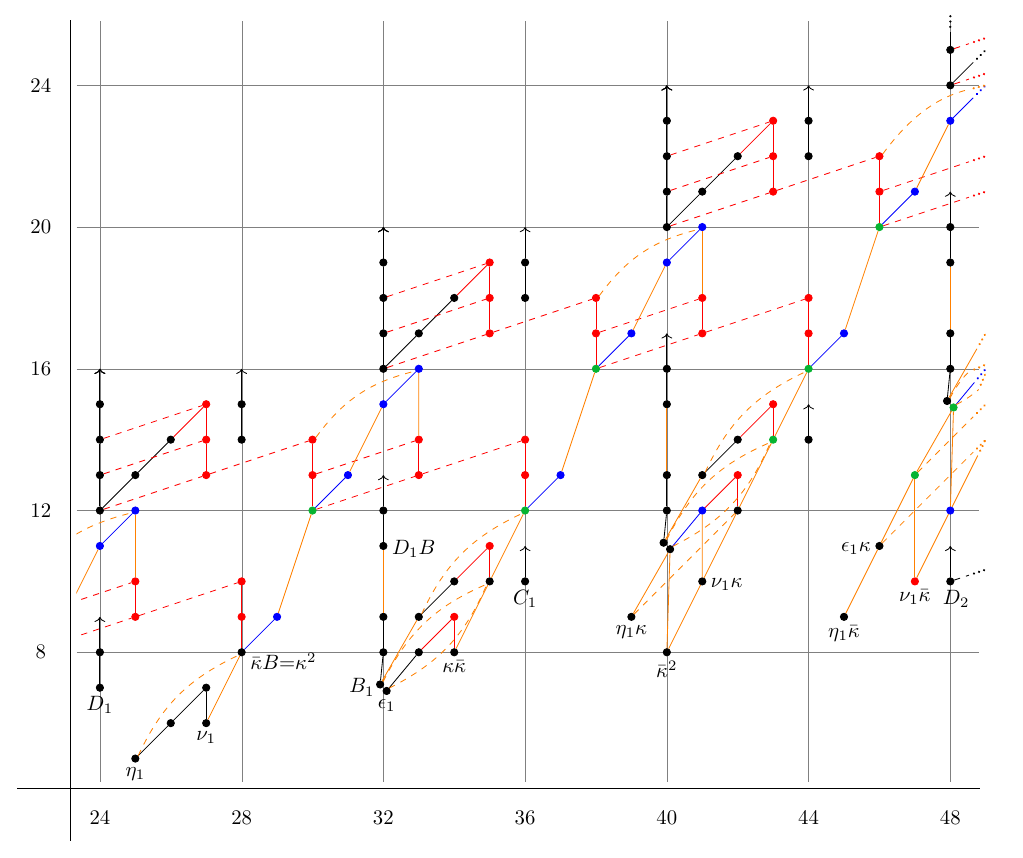}
\captionof{figure}{$\pi_{k,w}(tmf)$, $24\leq k\leq 48$}
\label{chart2}
\end{minipage}

\newpage

\begin{minipage}{.96\linewidth}
\includegraphics[width=\linewidth]{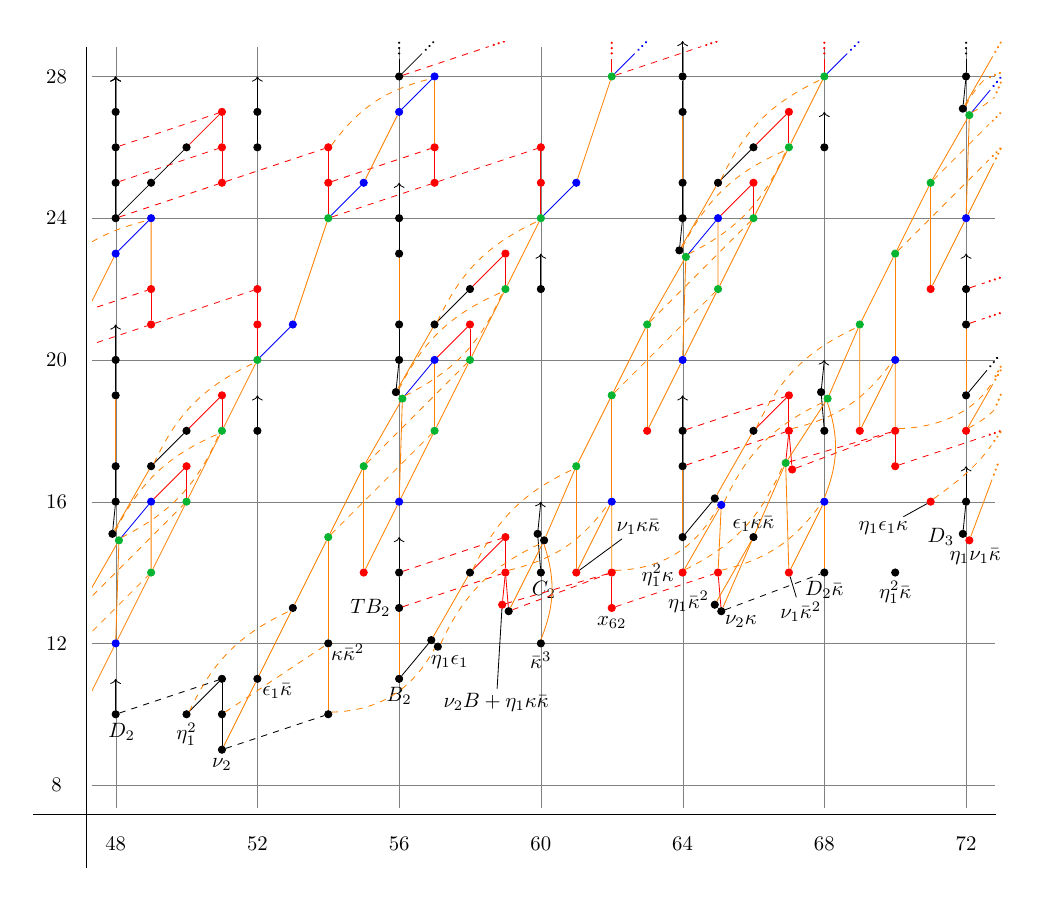}
\captionof{figure}{$\pi_{k,w}(tmf)$, $48\leq k\leq 72$}
\label{chart3}
\vspace{1cm}
\includegraphics[width=\linewidth]{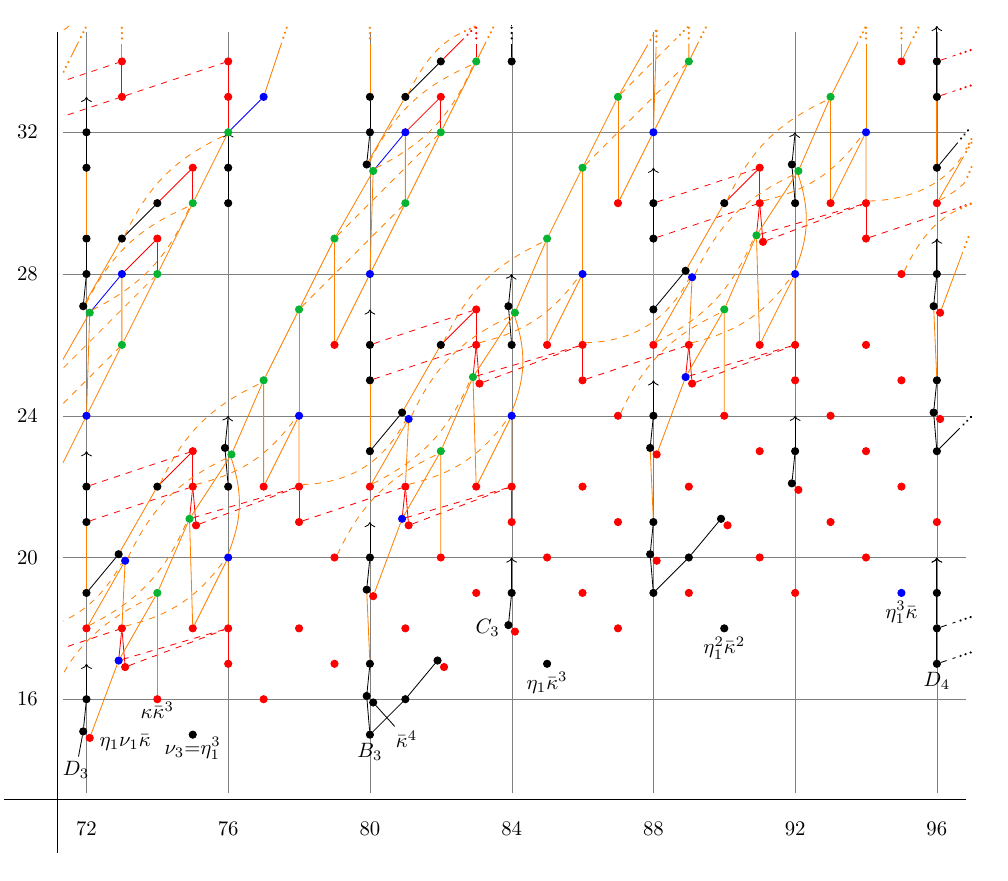}
\captionof{figure}{$\pi_{k,w}(tmf)$, $72\leq k\leq 96$}
\label{chart4}
\end{minipage}

\newpage

\begin{minipage}{.96\linewidth}
\includegraphics[width=\linewidth]{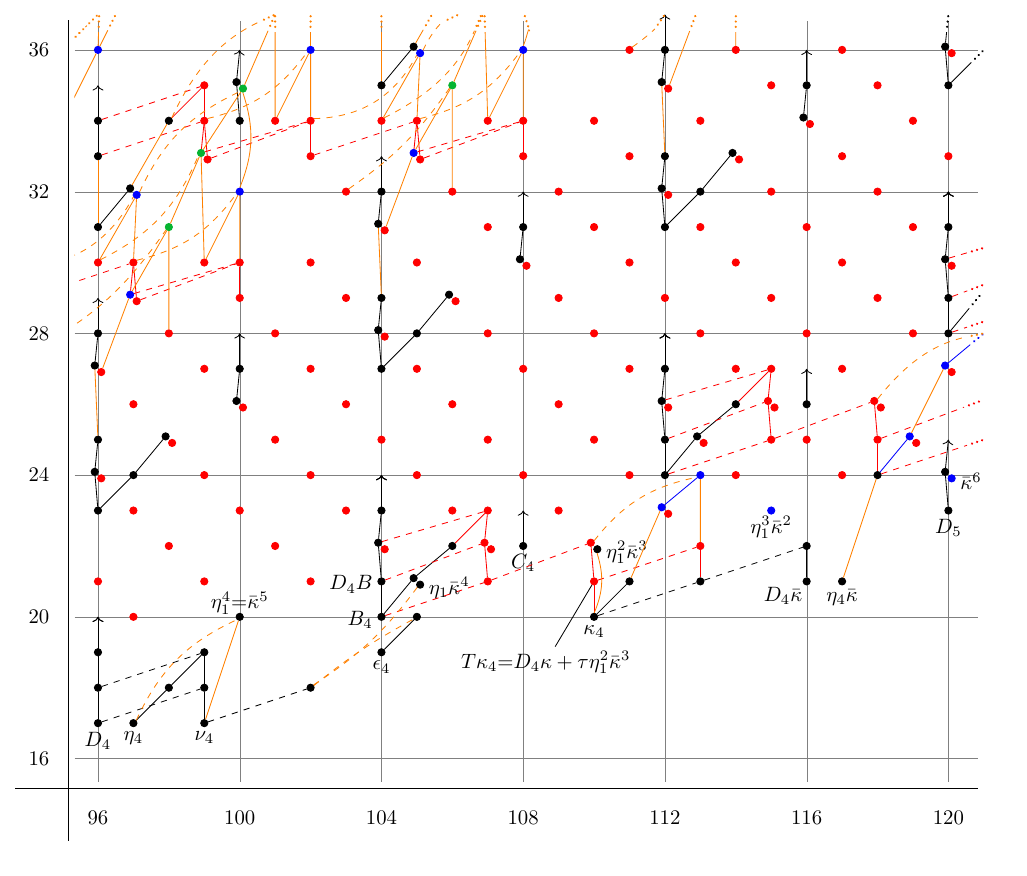}
\captionof{figure}{$\pi_{k,w}(tmf)$, $96\leq k\leq 120$}
\label{chart5}
\vspace{1cm}
\includegraphics[width=\linewidth]{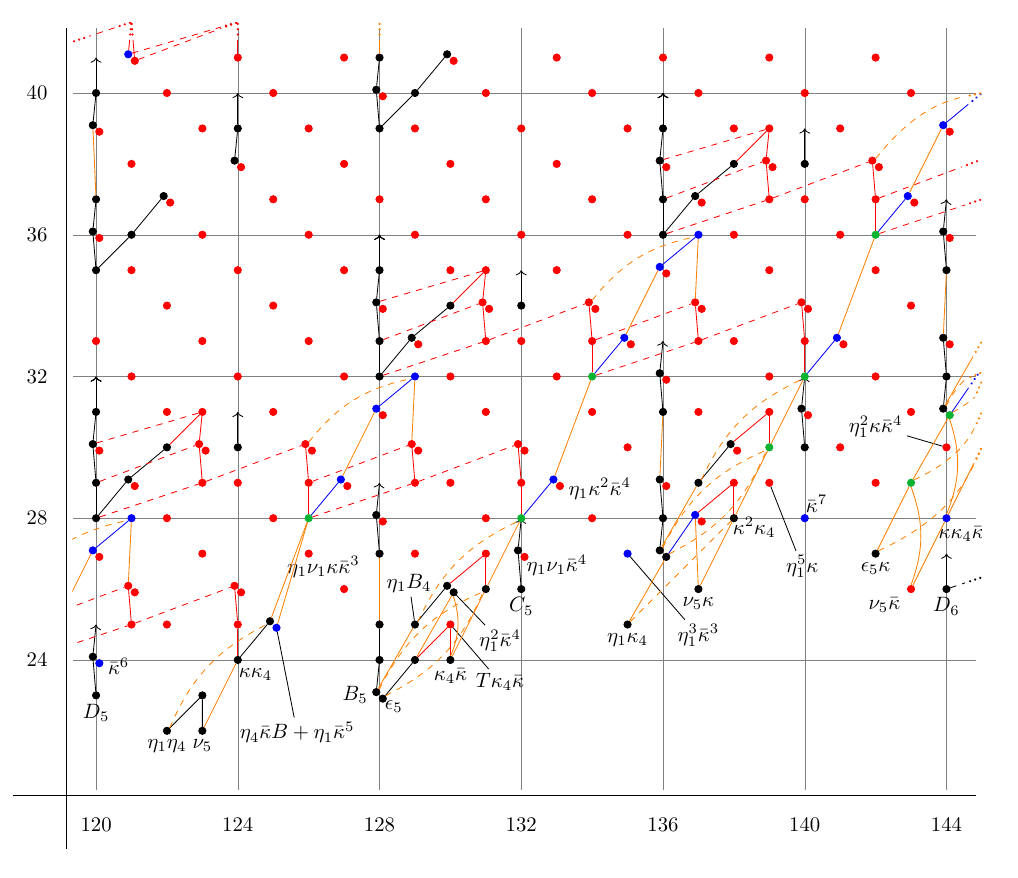}
\captionof{figure}{$\pi_{k,w}(tmf)$, $120\leq k\leq 144$}
\label{chart6}
\end{minipage}

\newpage

\begin{minipage}{.96\linewidth}
\includegraphics[width=\linewidth]{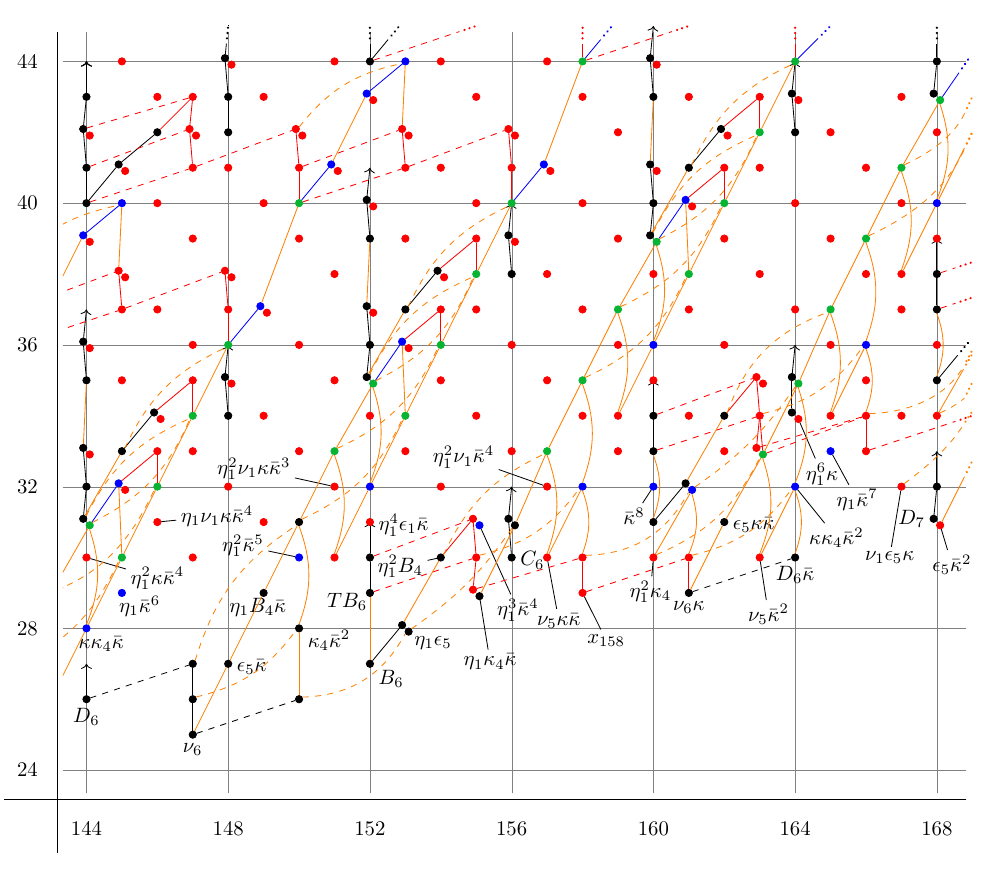}
\captionof{figure}{$\pi_{k,w}(tmf)$, $144\leq k\leq 168$}
\label{chart7}
\vspace{1cm}
\includegraphics[width=\linewidth]{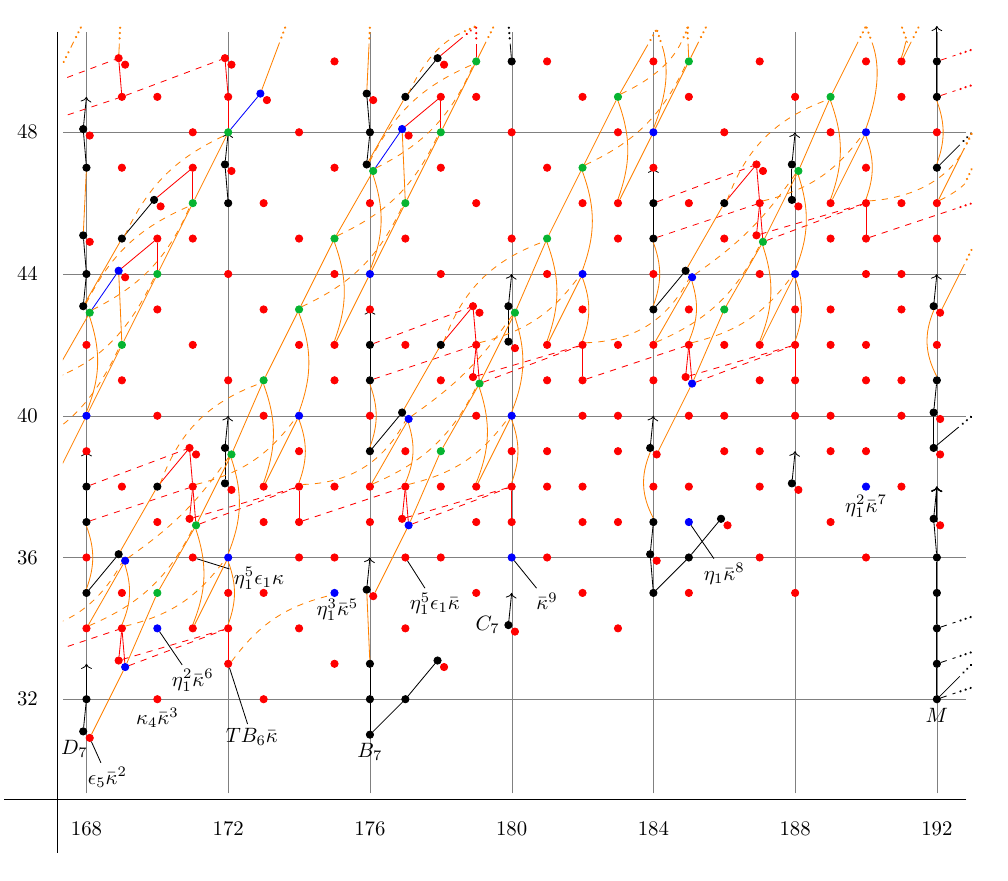}
\captionof{figure}{$\pi_{k,w}(tmf)$, $168\leq k\leq 192$}
\label{chart8}
\end{minipage}

\newpage

\begin{longtable}{lllllll}
\caption{$R_0'$-module generators of $\E_2$-page} \label{E2pagetable} \\
\hline
$t-s$ & $s$    & $w$    & $x$                 & $\mathrm{Ann}(x)$ & $d_2(x)$               & $d_2(xw_2)$                             \\ \hline
\endfirsthead
0     & 0      & 0      & 1                   & $(0)$             & 0                      & $\tau g\cdot\alpha\beta $                     \\
0     & 1      & 1      & $h_0$               & $(g^2)$           & 0                      & 0                                       \\
0     & 2      & 2      & $h_0^2$             & $(g^2)$           & 0                      & 0                                       \\
0     & $3+i$  & $3+i$  & $h_0^{3+i}$         & $(g)$             & 0                      & 0                                       \\
1     & 1      & 2      & $h_1$               & $(g^2)$           & 0                      & 0                                       \\
2     & 2      & 4      & $h_1^2$             & $(g)$             & 0                      & 0                                       \\
3     & 1      & 4      & $h_2$               & $(g)$             & 0                      & 0                                       \\
3     & 2      & 5      & $h_0h_2$            & $(g)$             & 0                      & 0                                       \\
\hline\hline

\end{longtable}

\newpage

\begin{longtable}{lllllll}
\caption{$R_1'$-module generators of $\E_3$-page} \label{E3pagetable} \\
\hline
$t-s$ & $s$    & $w$     & $x$                                                              & $\mathrm{Ann}(x)$  & $d_3(x)$                                                                                & $d_3(xw_2^2)$                                                                                  \\ \hline
0     & 0      & 0       & 1                                                                & $(\tau g^4w_1)$    & 0                                                                                       & $\tau^2g^2\cdot\beta g^2$                                                                      \\
0     & 1      & 1       & $h_0$                                                            & $(g^2,\tau gw_1)$  & 0                                                                                       & 0                                                                                              \\
0     & 2      & 2       & $h_0^2$                                                          & $(g^2,\tau g w_1)$ & 0                                                                                       & 0                                                                                              \\
0     & $3+i$  & $3+i$   & $h_0^{3+i}$                                                      & $(g)$              & 0                                                                                       & 0                                                                                              \\
1     & 1      & 2       & $h_1$                                                            & $(g^2)$            & 0                                                                                       & 0                                                                                              \\
\hline\hline
\end{longtable}

\begin{table}[ht]
\renewcommand{\arraystretch}{2.5}
\caption{Non-cyclic $R_1'$-modules of $\E_3$-page}
\label{noncycE3pagetable}
\begin{tabular}{|Sc|}
\hline \bfseries 
$\< x_1,x_2\rangle$                                                                                                                    \\ \hline
$\<\gamma,h_0e_0w_2\rangle \cong \dfrac{\Sigma^{5,30,30}R_1'\oplus\Sigma^{13,78,78}R_1'}{\<(0,g),(\tau g^2 w_1,\tau w_1)\rangle}$       \\
$\<d_0\gamma,h_2w_2\rangle \cong \dfrac{\Sigma^{9,48,48}R_1'\oplus\Sigma^{9,60,60}R_1'}{\<(0,g),(\tau g,\tau w_1)\rangle}$              \\
$\<e_0\gamma,h_0d_0w_2\rangle \cong\dfrac{\Sigma^{9,51,51}R_1'\oplus\Sigma^{13,75,75}R_1'}{\<(0,g^2),(\tau g,\tau)\rangle}$               \\
$\<\beta g^2,h_1^2\gamma w_2\rangle \cong \dfrac{\Sigma^{11,66,66}R_1'\oplus\Sigma^{15,90,90}R_1'}{\<(0,g),(\tau gw_1,\tau w_1)\rangle}$ \\

\hline
\end{tabular}
\end{table}

\newpage

\begin{longtable}{llllll}
\caption{$R_2'$-module generators of $\E_4$-page} \label{E4pagetable}\\
\hline
$t-s$ & $s$    & $w$     & $x$                          & $\mathrm{Ann}(x)$                        & $d_4(x)$                             \\ \hline
0     & 0      & 0       & 1                            & $(\tau g^4 w_1, \tau^2g^6,\tau^2g^2w_1)$ & 0                                    \\
0     & 1      & 1       & $h_0$                        & $(g^2,\tau gw_1)$                        & 0                                    \\
0     & 2      & 2       & $h_0^2$                      & $(g^2,\tau gw_1)$                        & 0                                    \\
0     & $3+i$  & $3+i$   & $h_0^{3+i}$                  & $(g)$                                    & 0                                    \\
1     & 1      & 2       & $h_1$                        & $(g^2,\tau^2gw_1)$                       & 0                                    \\
2     & 2      & 4       & $h_1^2$                      & $(g)$                                    & 0                                    \\
\hline\hline
\end{longtable}

\begin{table}[ht]
\renewcommand{\arraystretch}{2.5}
\caption{Non-cyclic $R_2'$-modules of $\E_4$-page}
\label{noncycE4pagetable}
\resizebox{12.6cm}{!}{
\begin{tabular}{|Sc|}
\hline
$\langle x_1,x_2\rangle$ or $\langle x_1,x_2,x_3\rangle$  \\ \hline

$\langle \gamma, h_0e_0w_2,h_1w_2^2\rangle \cong \dfrac{\Sigma^{5,30,30}R_2'\oplus\Sigma^{13,78,78}R_2'\oplus\Sigma^{17,114,114}R_2'}{\langle (0,g,0),(0,0,g^2),(\tau g^2w_1,\tau w_1,0),(\tau^2g^5,0,\tau^2gw_1),(\tau^2g^2w_1,0,0)\rangle}$ \\ \vspace{.5cm} 
\hspace{.6cm}$\langle d_0\gamma,h_2w_2\rangle \cong\dfrac{\Sigma^{9,48,48}R_2'\oplus\Sigma^{9,60,60}R_2'}{\langle (0,g),(\tau g,\tau w_1)\rangle}$                                                                                                 \\ 
\hspace{.6cm}$\langle \beta g^2,\gamma^3\rangle \cong\dfrac{\Sigma^{11,66,66}R_2'\oplus\Sigma^{15,90,90}R_2'}{\langle (g^2,g),(0,\tau w_1),(\tau^2g^2,0)\rangle}$ \\
$\langle \gamma w_1w_2^2,h_0e_0w_2^3\rangle \cong\dfrac{\Sigma^{25,154,154}R_2'\oplus\Sigma^{29,190,190}R_2'}{\langle (0,g),(\tau g^2,\tau w_1),(\tau^2g^2,0)\rangle}$                                                                 \\
$\langle d_0\gamma w_2^2,h_2w_2^3\rangle \cong\dfrac{\Sigma^{25,160,160}R_2'\oplus\Sigma^{25,172,172}R_2'}{\langle (0,g),(\tau g,\tau w_1)\rangle}$                                                                                    \\
$\langle \beta g^2w_1w_2^2,h_1^2\gamma w_2^3\rangle \cong\dfrac{\Sigma^{31,190,190}R_2'\oplus\Sigma^{31,202,202}R_2'}{\langle (0,g),(\tau g,\tau w_1)\rangle}$                                                                         \\  \hline
\end{tabular}
}
\end{table}

\newpage

\begin{longtable}{llllll}
\caption{$R_2'$-module generators of $\E_{\infty}$-page} \label{Einfpagetable}\\
\hline
$k=t-s$ & $s$    & $w$     & $x\in\Ext^{s,t,w}$        & $\mathrm{Ann}(x)$                                                                                & $y\in\pi_{k,w}(tmf)$ \\ \hline
0     & 0      & 0       & 1                              & \begin{tabular}[c]{@{}l@{}}$(\tau g^4w_1,\tau^2g^6,$\\ $\tau^2g^2w_1,\tau^3gw_1^2)$\end{tabular} & 1                              \\
0     & 1      & 1       & $h_0$                          & $(g^2,\tau gw_1)$                                                                                & $\ttwo$                        \\
0     & 2      & 2       & $h_0^2$                        & $(g^2,\tau gw_1)$                                                                                & $\ttwo^2$                      \\
0     & $3+i$  & $3+i$   & $h_0^{3+i}$                    & $(g)$                                                                                            & $\ttwo^{3+i}$                  \\
1     & 1      & 2       & $h_1$                          & $(g^2,\tau^2gw_1)$                                                                               & $\eta$                         \\
2     & 2      & 4       & $h_1^2$                        & $(g)$                                                                                            & $\eta^2$                       \\
\end{longtable}

\begin{table}[ht]
\caption{Non-cyclic $R_2'$-modules of $\E_{\infty}$-page}
\label{noncycEinfpagetable}
\renewcommand{\arraystretch}{2.5}
%
}}$ \\
$\langle d_0\gamma,h_2w_2\rangle \cong\dfrac{\Sigma^{9,48,48}R_2'\oplus\Sigma^{9,60,60}R_2'}{\langle (0,g),(\tau g,\tau w_1),(\tau^3w_1,0)\rangle}$                                                                                                    \\
$\langle \gamma w_1w_2^2,h_0e_0w_{2}^3\rangle \cong\dfrac{\Sigma^{25,154,154}R_2'\oplus\Sigma^{29,190,190}R_2'}{\langle (0,g),(\tau^2g^2,0),(\tau g^2,\tau w_1),(\tau^3gw_1,0)\rangle}$                                                                  \\
$\langle d_0\gamma w_2^2,h_2w_2^3\rangle \cong\dfrac{\Sigma^{25,160,160}R_2'\oplus\Sigma^{25,172,172}R_2'}{\langle (0,g),(\tau g,\tau w_1),(\tau^3w_1,0)\rangle}$ \\ \hline
\end{tabular}
\end{table}

\newpage

\begin{longtable}{lllll}
\caption{Algebra generators of $\E_{\infty}$-page and $\pi_{*,*}tmf$} \label{Einfalggentable} \\
\hline
$k=t-s$ & $s$ & $w$ & $x\in\E_{\infty}^{s,t,w}$             & $y\in\pi_{k,w}(tmf)$ \\ \hline
0                            & 1   & 1   & $h_0$                          & $\ttwo$                 \\
1                            & 1   & 2   & $h_1$                          & $\eta$                  \\
3                            & 1   & 4   & $h_2$                          & $\nu$                   \\
8                            & 3   & 11  & $c_0$                          & $\epsilon$              \\
8                            & 4   & 12  & $w_1$                          & $B$                     \\
12                           & 6   & 18  & $h_0^3\alpha$                  & $C$                     \\
\hline\hline
\end{longtable}

\newpage

\begin{longtable}{lllll}
\caption{Module generators of $\pi_{*,*}tmf$}
\label{homotopymodgentable} \\
\hline
$k=t-s$ & $s$    & $w$     & $x\in\E_{\infty}^{s,t,w}$        & $y\in\pi_{k,w}(tmf)$ \\ \hline
0       & 0      & 0       & 1                              & 1                            \\
8       & 3      & 11      & $c_0$                          & $\epsilon$                   \\
12      & $6+i$  & $18+i$  & $h_0^{3}\alpha$                & $C$                          \\
14      & 4      & 18      & $d_0$                          & $\kappa$                     \\
24      & $7+i$  & $31+i$  & $h_0\alpha^2$                  & $D_1$                        \\
25      & 5      & 30      & $\gamma$                       & $\eta_1$                     \\
\hline\hline
\end{longtable}

\newpage

\begin{longtable}{lllll}
\caption{$\kappabar$-multiplications in $\pi_{*,*}tmf$} \label{kappabarextensiontable}\\
\hline
$k$ & $w$     & $y\in\pi_{k,w}(tmf)$       & $\kappabar$-multiplication                                                                                                                                                                                                                          & Proof                                                                                                                                                      \\ \hline
0   & 0       & 1                          & $\kappabar^n\neq 0,n\geq 0$                                                                                                                                                                                                                    & $\E_{\infty}$                                                                                                                                                        \\
0   & 1       & $\ttwo$                    & \begin{tabular}[c]{@{}l@{}}$\ttwo\cdot\kappabar^2=\tau^2\epsilon_1B$\\ $\ttwo\cdot\kappabar^3=\tau^2\epsilon_1\kappabar B$\\ $\ttwo\cdot\kappabar^4=0$\end{tabular}                                                                            & \begin{tabular}[c]{@{}l@{}}classical\\ classical\\ classical\end{tabular}                                                                                  \\
0   & 2       & $\ttwo^2$                  & $\ttwo^2\cdot\kappabar^2=0$                                                                                                                                                                                                                    & classical                                                                                                                                                  \\
\hline\hline
\end{longtable}

\newpage

\begin{longtable}{lllll}
\caption{$\ttwo$-extensions in $\pi_{*,*}tmf$} \label{ttwoextensiontable}\\
\hline
$k$ & $w$     & $y\in\pi_{k,w}(tmf)$ & $\ttwo$-extension                                                                                                                                                                                                            & Proof                                                                                                                                                     \\ \hline
0   & 0       & 1                    & \begin{tabular}[c]{@{}l@{}}$\ttwo\cdot\kappabar^2=\tau^2\epsilon_1 B$\\ $\ttwo\cdot\kappabar^3=\tau^2\epsilon_1\kappabar B$\\ $\ttwo\cdot\kappabar^4=0$\end{tabular}                                                         & \begin{tabular}[c]{@{}l@{}}classical\\ classical\\ classical\end{tabular}                                                                                 \\
0   & 1       & $\ttwo$              & $\ttwo\cdot\ttwo\kappabar^2=0$                                                                                                                                                                                               & classical                                                                                                                                                 \\
0   & 2       & $\ttwo^2$            & $\ttwo\cdot\ttwo^2\kappabar=0$                                                                                                                                                                                               & classical                                                                                                                                                 \\
\hline\hline
\end{longtable}

\newpage

\begin{longtable}{lllll}
\caption{$\eta$-extensions in $\pi_{*,*}tmf$} \label{etaextensiontable}\\
\hline
$k$ & $w$     & $y\in\pi_{k,w}(tmf)$       & $\eta$-extension                                     & Proof  \\ \hline
0   & 0       & 1                          & \begin{tabular}[c]{@{}l@{}}$\eta\cdot\kappabar^2=\tau\nu_1\kappa$\\ $\eta\cdot\kappabar^3=0$\end{tabular}                                                                                                                                                                     & \begin{tabular}[c]{@{}l@{}}classical\\ deg\end{tabular}                                                                                                                    \\
0   & 1       & $\ttwo$                    &                                                                                                                                                                                                                                                                               &                                                                                                                                                                            \\
0   & 2       & $\ttwo^2$                  &                                                                                                                                                                                                                                                                               &                                                                                                                                                                            \\
\hline\hline
\end{longtable}

\newpage

\begin{longtable}{lllll}
\caption{$\nu$-extensions in $\pi_{*,*}tmf$} \label{nuextensiontable}\\
\hline
$k$ & $w$     & $y\in\pi_{k,w}(tmf)$     & $\nu$-extension                                          & Proof                                      \\ \hline
0   & 0       & 1                        &                                                          &                                            \\
0   & 1       & $\ttwo$                  &                                                          &                                            \\
0   & 2       & $\ttwo^2$                &                                                          &                                            \\
0   & $3+i$   & $\ttwo^{3+i}$            &                                                          &                                            \\
1   & 2       & $\eta$                   &                                                          &                                            \\
2   & 4       & $\eta^2$                 &                                                          &                                            \\
\hline\hline
\end{longtable}

\begin{table}[ht]
\caption{Other relations in $\pi_{*,*}tmf$} \label{otherrelationstable}
\begin{tabular}{llll}
\hline
$k$ & $w$ & Relation in $\pi_{k,w}(tmf)$                                                                & Proof                                      \\ \hline
105 & 124 & $\nu^2\nu_4=\tau\eta\epsilon_4+\tau^2\eta_1\kappabar^4$                                & classical                                  \\
110 & 131 & $D_4\kappa=\ttwo\kappa_4=T\kappa_4+\tau \eta_1^2\kappabar^3$                           & Lemma~\ref{r.110.21} \\
130 & 155 & $D_4\kappa\kappabar=\ttwo\kappa_4\kappabar=T\kappa_4\kappabar+\tau\eta_1^2\kappabar^4$ & Lemma~\ref{r.130.25} \\
130 & 155 & $\eta^2\epsilon_5=T\kappa_4\kappabar+\tau\eta_1^2\kappabar^4$                           & Lemma~\ref{r.130.25} \\ \hline\hline
\end{tabular}
\end{table}

\newpage

\printbibliography

\end{document}